\documentclass[notitlepage,11pt,reqno]{amsart}
\usepackage[foot]{amsaddr}
\usepackage{amssymb,nicefrac,bm,upgreek,mathtools,verbatim}
\usepackage[final]{hyperref}
\usepackage[mathscr]{eucal}
\usepackage{dsfont}
\usepackage[normalem]{ulem}
\usepackage{amsopn,esint}
\usepackage[T1]{fontenc}

\usepackage[utf8]{inputenc}
\usepackage{apptools}
\AtAppendix{\counterwithin{lemma}{section}} 

\usepackage[margin=1in]{geometry}
\newcommand{\stkout}[1]{\ifmmode\text{\sout{\ensuremath{#1}}}\else\sout{#1}\fi}

\newtheorem{lemma}{Lemma}[section]
\newtheorem{theorem}{Theorem}[section]
\newtheorem{proposition}{Proposition}[section]

\theoremstyle{definition}
\newtheorem{definition}{Definition}[section]
\newtheorem{assumption}{Assumption}[section]

\newtheorem{example}{Example}[section]

\theoremstyle{remark}
\newtheorem{remark}{Remark}[section]
\numberwithin{theorem}{section}
\numberwithin{equation}{section}

\hypersetup{
  colorlinks=true,
  citecolor=mblue,
  linkcolor=mblue,
  urlcolor = blue,
  anchorcolor = blue,
  frenchlinks=false,
  pdfborder={0 0 0},
  naturalnames=false,
  hypertexnames=false,
  breaklinks}
\usepackage[capitalize,nameinlink]{cleveref}

\usepackage[abbrev,msc-links,nobysame,citation-order]{amsrefs}

\crefname{section}{Section}{Sections}
\crefname{subsection}{Section}{Sections}
\crefname{condition}{Condition}{Conditions}
\crefname{hypothesis}{Hypothesis}{Conditions}
\crefname{assumption}{Assumption}{Assumptions}
\crefname{lemma}{Lemma}{Lemmas}
\crefname{fact}{Fact}{Facts}

\Crefname{figure}{Figure}{Figures}

\crefformat{equation}{\textup{#2(#1)#3}}
\crefrangeformat{equation}{\textup{#3(#1)#4--#5(#2)#6}}
\crefmultiformat{equation}{\textup{#2(#1)#3}}{ and \textup{#2(#1)#3}}
{, \textup{#2(#1)#3}}{, and \textup{#2(#1)#3}}
\crefrangemultiformat{equation}{\textup{#3(#1)#4--#5(#2)#6}}%
{ and \textup{#3(#1)#4--#5(#2)#6}}{, \textup{#3(#1)#4--#5(#2)#6}}%
{, and \textup{#3(#1)#4--#5(#2)#6}}

\Crefformat{equation}{#2Equation~\textup{(#1)}#3}
\Crefrangeformat{equation}{Equations~\textup{#3(#1)#4--#5(#2)#6}}
\Crefmultiformat{equation}{Equations~\textup{#2(#1)#3}}{ and \textup{#2(#1)#3}}
{, \textup{#2(#1)#3}}{, and \textup{#2(#1)#3}}
\Crefrangemultiformat{equation}{Equations~\textup{#3(#1)#4--#5(#2)#6}}%
{ and \textup{#3(#1)#4--#5(#2)#6}}{, \textup{#3(#1)#4--#5(#2)#6}}%
{, and \textup{#3(#1)#4--#5(#2)#6}}

\crefdefaultlabelformat{#2\textup{#1}#3}
\newcommand{\vertiii}[1]{{\left\vert\kern-0.25ex\left\vert\kern-0.25ex\left\vert #1
    \right\vert\kern-0.25ex\right\vert\kern-0.25ex\right\vert}}

\newcommand{\cH}{{\mathcal{H}}}  
\newcommand{\sJ}{{\mathscr{J}}}
\newcommand{\sR}{\mathscr{R}}

\newcommand{\RR}{\mathds{R}}

\newcommand{\Rn}{{\mathds{R}^{n}}}

\newcommand{\D}{\mathrm{d}}
\newcommand{\E}{\mathrm{e}}
\newcommand{\Ind}{\mathds{1}}   

\newcommand{\df}{\coloneqq}

\DeclareMathOperator*{\dist}{dist}
\DeclareMathOperator*{\diam}{diam}

\DeclareMathOperator*{\supp}{support}

\newcommand{\grad}{\nabla}

\newcommand{\abs}[1]{\lvert#1\rvert}
\newcommand{\norm}[1]{\lVert#1\rVert}

\usepackage{color}
\definecolor{dmagenta}{rgb}{.4,.1,.5}
\definecolor{dblue}{rgb}{.0,.0,.5}
\definecolor{mblue}{rgb}{.0,.0,.7}
\definecolor{ddblue}{rgb}{.0,.0,.4}
\definecolor{dred}{rgb}{.7,.0,.0}
\definecolor{dgreen}{rgb}{.0,.5,.0}
\definecolor{Eeom}{rgb}{.0,.0,.5}

\allowdisplaybreaks

\newcommand{\ttl}{\Large Boundary regularity of mixed local-nonlocal operators and its application}
\begin{document}
\title[Boundary regularity]
{\ttl}

\author{Anup Biswas, Mitesh Modasiya and Abhrojyoti Sen}

\address{
Department of Mathematics, Indian Institute of Science Education and Research, Dr. Homi Bhabha Road,
Pune 411008, India. Email: anup@iiserpune.ac.in; mitesh.modasiya@students.iiserpune.ac.in; abhrojyoti.sen@acads.iiserpune.ac.in }

\begin{abstract}
Let $\Omega$ be a bounded $C^2$ domain in $\Rn$ and
$u\in C(\Rn)$ solves 
\begin{equation*}
\begin{aligned}
\Delta u + a Iu + C_0|Du| \geq -K\quad \text{in}\; \Omega,
\quad
\Delta u + a Iu - C_0|Du|\leq K \quad \text{in}\; \Omega,
\quad
u=0\quad \text{in}\; \Omega^c,
\end{aligned}
\end{equation*}
in viscosity sense, where $0\leq a\leq A_0$, $C_0, K\geq 0$, and $I$ is a suitable nonlocal operator.
We show that $u/\delta$ is in $C^{\upkappa}(\bar \Omega)$ for some $\upkappa\in (0,1)$, where $\delta(x)=\dist(x, \Omega^c)$. Using this result we also establish that 
$u\in C^{1, \gamma}(\bar\Omega)$. Finally, we apply these results
to study an overdetermined problem for mixed local-nonlocal operators.
\end{abstract}

\maketitle


\section{Introduction and results}

We are interested in the integro-differential operator $L$  
of the form 
$$L u =\Delta u + a Iu,$$
where $0\leq a\leq A_0$, and $I$ is a nonlocal operator given by
\begin{align*}
Iu(x) &=  \int_{\Rn} \Big( u(x+y) - u(x) - \Ind_{\{|y|\leq 1\}}y \cdot \grad u(x) \Big)  k(y) \D{y}  
\\
&= \frac{1}{2} \int_{\Rn} (u(x+y) + u(x-y) - 2u(x) ) k(y) \D{y}
\end{align*}
for some nonnegative, symmetric kernel $k$, that is, $k(y)=k(-y)\geq 0$ for all $y$. Operator $L$ appears as a generator of a L\'{e}vy process which is obtained by superimposing a Brownian motion, running twice as fast as standard $n$-dimensional Brownian motion, and an independent pure-jump L\'evy process corresponding
to the nonlocal operator $aI$.
Throughout 
this article we impose the following assumptions on the kernel 
$k$.

\begin{assumption}\label{Assmp-1}
\hspace{-2em}
\begin{itemize}
\item[(a)] For some $\alpha \in (0,2)$ and a 
kernel $\widehat{k}$ we have that, for all $r\in (0, 1]$,
$$
r^{n + \alpha} k(r y) \leq \widehat{k}(y) \quad \text{for}\; y\in\Rn,
$$
and $\widehat{k}(y) = \frac{\Lambda}{\abs y^{n +\alpha}} \Ind_{B_1}(y) + J(y) \Ind_{B_1^c}(y)$, where $\Lambda>0$ and
$$\int_{B_1^c} J(y) \D{y} < \infty.$$

\item[(b)] There exists a $\beta>0$ such that
for any $r\in (0, 1]$ and $x_0 \in \Rn$
the following holds: for all $x,y \in B_{\frac{r}{2}}(x_0)$ and 
$z\in B_r^c(x_0)$ we have
$$
k(x-z) \leq  \varrho\, k(y-z)\quad \text{for}\; |y-z|< \beta,
$$
for some  $\varrho >1$.
\end{itemize}
\end{assumption}
\cref{Assmp-1}(a) will be used to study certain scaled operators 
and to find the exact behaviour of $I\delta(x)$ near the boundary
where $\delta(x)$ denotes the distance function from the complement 
$\Omega^c$. \cref{Assmp-1}(b)
will be used to apply the Harnack estimate from \cite{F09}. It should be noted that \cite{F09} uses a stronger hypothesis compared to \cref{Assmp-1}(b).
\cref{Assmp-1} is satisfied by a large class of nonlocal kernels
as shown in the examples below.

\begin{example}
The following class of nonlocal kernels satisfy \cref{Assmp-1}.
\begin{itemize}
\item[(i)] $k(y)= \frac{1}{|y|^{n+\alpha}}$ for some 
$\alpha\in (0,2)$. More generally, we may take
$k(y)= \frac{1}{|y|^{n+\alpha}}\Ind_{B}(y)$ for some
ball $B$ centered at the origin.
\item[(ii)] $k(y)\asymp \frac{\Psi(|y|^{-2})}{|y|^n }$ where 
$\Psi$ is Bernstein function vanishing at $0$. In particular, 
$\Psi$ is strictly increasing and concave. These class of nonlocal kernels
correspond to a special class of L\'evy processes, known as
subordinate Brownian motions (see \cite{SSV}). Assume that 
$\Psi$ satisfies a global weak scaling property with parameters
$\mu_1, \mu_2\in (0,1)$, that is,
$$\lambda^{\mu_1}\Psi(s)\lesssim \Psi(\lambda s)\lesssim \lambda^{\mu_2}\Psi(s)\quad \text{for}\; s>0, \lambda\geq 1.$$
Then it is easily seen that
\begin{align*}
r^{n+2\mu_2} k(ry)\asymp r^{2\mu_2}\frac{\Psi(|ry|^{-2})}{|y|^n}
\lesssim \frac{\Psi(|y|^{-2})}{|y|^n}
\lesssim \Ind_{B_1}(y) \frac{\Psi(1)}{|y|^{n+2\mu_2}}
+ \Ind_{B^c_1}(y) \frac{\Psi(1)}{|y|^{n+2\mu_1}}
\end{align*}
for all $r\in (0,1]$. Thus \cref{Assmp-1}(a) holds. Using the weak scaling property we can also check that \cref{Assmp-1}(b) holds.
\end{itemize}
\end{example}
Let $\Omega$ be a bounded $C^2$ domain in $\Rn$.
Let $u\in C(\Rn)$ be a viscosity solution to
\begin{equation}\label{Eq-1}
\begin{split}
L u + C_0 |Du| & \geq -K \; \quad \text{in}\; \Omega,
\\
Lu - C_0 |Du| & \leq K  \quad \quad\text{in}\; \Omega,
\\
u & =0 \; \quad \quad\text{in}\; \Omega^c,
\end{split}
\end{equation}
for some nonnegative constants $C_0, K$. Though the results in this article are obtained for viscosity solutions, the results are also 
applied for weak solutions, see \cref{R1.2} below for more details.
Our equations in \eqref{Eq-1} are motived by the operators of the form
\begin{equation}\label{HJB}
L u + H(Du, x):= L u + \inf_{\mu}\sup_{\nu}
\{b_{\mu,\nu}(x)\cdot Du(x) + f_{\mu,\nu}(x)\}=0
\quad \text{in}\; \Omega,\quad u=0\quad \text{in}\; \Omega^c.
\end{equation}
Such equations arise in the study of stochastic control problems
where the control can influence the dynamics only through the drift
$b_{\mu,\nu}$. 

Integro-differential operators involving both local and nonlocal operators
 have gained interest very recently. See for instance, 
 \cite{AR22,Barles08,Barles08a,BDVV,BDVV22,BVDV,
 FSZ22,GM02,GK22,M19,MZ21}.
For linear equations, one may recover up to the boundary $C^{1, \gamma}$ regularity of $u$ from the $W^{2, p}$ regularity (cf. \cite[Theorem~3.1.22]{GM02}). 
Recently, inspired by \cite{CS09,CS11}, 
 interior regularity of the solutions of \eqref{HJB} are studied in \cite{M19,MZ21}. Let us also mention the recent works \cite{DM22,GL22} where interior H\"{o}lder regularity of the
 gradient is established for the weak solutions of degenerate
 elliptic equations of mixed type. 
In this article, we are interested in up to the boundary regularity of the solutions. It should also be noted that we are dealing with
inequalities.

Our first result deals with the Lipschitz regularity of the solution.
\begin{theorem}\label{T2.1}
Suppose that \cref{Assmp-1}(a) holds and $u\in C(\Rn)$ is
a viscosity solution to \eqref{Eq-1}.
Then, for some constant $C$, dependent only on  $n, \Omega, C_0, A_0, \widehat{k}$, we have 
\begin{equation}\label{ET2.1A}
\norm{u}_{C^{0,1}(\Rn)} \leq C K.
\end{equation}
\end{theorem}
Proof of \cref{T2.1} is based on two standard ingredients, interior
estimates from \cite{MZ21} and barrier function. It can be easily shown that $\dist(\cdot, \Omega^c)$ gives a barrier function at
the boundary (cf. \cref{L2.1}). 
Next we investigate finer regularity property of $u$ near $\partial\Omega$.
Let $\delta(x)=\dist(x, \Omega^c)$  be the distance function from the boundary. Modifying $\delta(x)$ inside $\Omega$, if required, we may assume that $\delta\in C^2(\bar\Omega)$ 
(cf. \cite[Theorem~5.4.3]{DZ11}).
Our next result establishes H\"{o}lder regularity of $u/\delta$ up to the boundary.

\begin{theorem}\label{T2.2}
Suppose that \cref{Assmp-1} holds. Let $u$ be a viscosity
solution to \eqref{Eq-1}.
 Then there exists 
$\upkappa\in (0,(2-\alpha)\wedge 1)$ such that 
\begin{equation}\label{BMS01}
\norm{u/\delta}_{C^{\upkappa}(\bar{\Omega})}\leq C_1 K,
\end{equation}
for some constant $C_1$, where
$\upkappa, C_1$ depend on $n, C_0, A_0, \widehat{k}, \Omega$.
\end{theorem}
The regularity of $\partial\Omega$ in the above result can be relaxed to $C^{1,1}$. See \cref{R2.3} for more detail.
For elliptic operators similar estimate is obtained by 
Krylov \cite{K83}. Boundary estimate for fractional Laplacian 
operators are studied by Ros-Oton and Serra in \cite{RS14,RS16,RS17}. Result of \cite{RS14} has been
extended for nonlocal operators
with kernel of variable orders by Kim et al. \cite{KKLL}
whereas extension to the fractional $p$-Laplacian operator can be found
in \cite{IMS}.
For the proof of \cref{T2.2} we follow the approach of \cite{RS14}
which is inspired by a method of Caffarelli \cite[p.~39]{Kaz}.
A key step in this analysis is the oscillation lemma (see \cref{P2.1}) for $u/\delta$ which involves computation of
$L((u- \kappa \delta)^+)$ for some suitable constant $\kappa$.
Note that, by 
\cref{T2.1}, $Iu$ is bounded in $\Omega$ for $\alpha\in (0,1)$, and therefore, in this case we can follow standard approach of local operators to get the estimate \eqref{BMS01}. But for $\alpha\in[1,2)$, $I\delta$ becomes singular near $\partial \Omega$.
So we have to do several careful estimates to apply the method of 
\cite{RS14}.

Using \cref{T2.2} we establish boundary regularity of $Du$
(compare it with Fall-Jarohs \cite{FJ21}). This is
the content of our next result.
\begin{theorem}\label{T2.3}
Let \cref{Assmp-1} hold.
There exist constants $\gamma, C$, dependent on $\Omega, C_0, A_0, n,\widehat{k}$, such that for any solution $u$ of \eqref{Eq-1} 
we have
\begin{equation}\label{ET2.3A}
\norm{u}_{C^{1, \gamma}(\bar\Omega)}\leq CK.
\end{equation}
\end{theorem}
\begin{remark}
By the dependency of the constants in \cref{T2.1,T2.2,T2.3} on 
$\widehat{k}$ we mean the dependency on $\alpha, \Lambda$ and
$\int_{|y|\geq 1} J(y)\D{y}$.
\end{remark}
To cite a specific application of the above results, 
let us consider $u, v\in C(\Rn)$ satisfying 
\begin{align*}
Lu + H_1(Du, x) &=0, \quad \text{in}\; \Omega,\quad u=0\quad\text{in}\; \Omega^c,
\\
Lv + H_2(Dv, x) &=0, \quad \text{in}\; \Omega,\quad v=0\quad\text{in}\; \Omega^c,
\end{align*}
respectively. If $|H_1(p, x)-H_2(q, x)|\leq C_0 |p-q| +K$
for all $p, q\in\Rn$ and $x\in \Omega$, then using the interior
regularity of $u, v$ from \cite{MZ21} and the coupling result
\cite[Theorem~5.1]{BM21} it can be easily seen that $w=u-v$ satisfies
\eqref{Eq-1}. Our result \cref{T2.3} then gives a $C^{1,\gamma}$
estimate of $w$ up to the boundary. The above results 
can also be used to establish anti-maximum principle for the
generalized principal eigenvalues of nonlinear operators of the form \eqref{HJB} (cf. \cite{B20,BL21,CP79}).

\begin{remark}\label{R1.2}
 Though the above results are mentioned for viscosity solutions, Theorem~\ref{T2.1}-\ref{T2.3} can also be applied for weak solutions (at least for equations). To see this, let us assume that
$\Omega$ be a $C^{2, \kappa}$ domain, $\kappa\in (0, 1)$.
Suppose that
for some given Lipschitz function $f:\Rn\to\RR$, there exists a unique weak solution 
$u\in H^1_0(\Omega)$ to 
\begin{equation}\label{E1.6}
Lu + f(Du)= g\quad \text{in}\; \Omega, \quad u=0\quad \text{in}\; \Omega^c,
\end{equation}
for every $g\in L^\infty(\Omega)$.
Now consider a sequence of smooth mollifications $g_\varepsilon$ of $g$ such that $\sup_{\Omega}|g_\varepsilon-g|\to 0$, as $\varepsilon\to 0$. Let $u_\varepsilon$ be the unique weak solution to \eqref{E1.6} corresponding to $g_\varepsilon$. Since, by Sobolev 
embedding $u_{x_i}\in L^p(\Omega)$ for $p\in [1,\frac{2n}{n-2}]$,
applying \cite[Theorem~3.1.22]{GM02} and a bootstrapping argument we 
obtain that for some $p>n$
\begin{equation}\label{E1.7}
\norm{u_\varepsilon}_{W^{2,p}(\Omega)}\leq C\left(1+ \norm{Du_\varepsilon}_{L^2(\Omega)} + \norm{g_\varepsilon}_{L^\infty(\Omega)}\right),
\end{equation}
for some constant $C$ independent of $u_\varepsilon$. This, of course, implies $u_\varepsilon\in C^{1, \gamma}(\bar\Omega)$. Applying
\cite[Theorem~3.1.12]{GM02} we have 
$u_\varepsilon\in C^{2, \gamma}(\bar\Omega)$ and therefore, $u_\varepsilon$ is a viscosity solution to \eqref{E1.6} when $g$ is replaced by $g_\varepsilon$. Hence we can apply Theorem~\ref{T2.1}-\ref{T2.3} on $u_\varepsilon$. In particular, 
$$
\sup_{\varepsilon\in (0,1)}\norm{Du_\varepsilon}_{L^\infty(\Omega)}<\infty.
$$ 
Now, using the stability estimate \eqref{E1.7}, we can pass the limit, as $\varepsilon\to 0$, to show that $u_\varepsilon\to u$ where $u$ is the weak solution to \eqref{E1.6} with data $g$ and $u$ also satisfies the estimates in Theorem~\ref{T2.1}-\ref{T2.3}.

\end{remark}

Next we apply \cref{T2.2,T2.3} to study an overdetermined problem.
More precisely, we consider a solution $u$ to the problem
\begin{equation}\label{BMS02}
\begin{split}
Lu + H(|D u|) & =f(u)\quad  \text{in} \; \Omega,
\\
u=0\quad \text{in} \;  \Omega^c,\quad u&>0\quad \text{in}\; \Omega,
\quad \frac{\partial u}{\partial \rm n}=c\quad \text{on}\; \partial \Omega,
\end{split}
\end{equation}
where ${\rm n}$ is the unit inward normal and $H:\RR\to\RR, f:\Rn\to\RR$ are locally Lipschitz. In \cref{T3.1} we show that $\Omega$ must be ball, provided the nonlocal kernel $k$ satisfies certain conditions. Overdetermined problem was first studied in the celebrated work of Serrin \cite{S71} where it was shown that
if there exists a positive solution to
$$-\Delta u =1\quad \text{in}\; \Omega, 
\quad u=0,\; \frac{\partial u}{\partial \rm n}=c\quad \text{on}\; \partial \Omega,$$
then $\Omega$ must be a ball. Serrin's work has been generalized
for a vast class of operators, see for instance, 
\cite{BJ20,BL21,CiSl09,FJ15,FK08,FMW17,FV10,SV19}. In this article, we follow the 
method of \cite{BJ20,FJ15} to establish our result on overdetermined problem concerning \eqref{BMS02}.

The rest of the article is organized as follows: In the next
\cref{Proofs} we provide the proofs of \cref{T2.1,T2.2,T2.3} and
\cref{Overdetermined} discusses the overdetermined problems.

\section{Proofs of \cref{T2.1,T2.2,T2.3}}\label{Proofs}

We begin by showing that $\delta$ is a barrier function to $u$
at the boundary.  
\begin{lemma}\label{L2.1}
Suppose that \cref{Assmp-1}(a) holds and $u$ be a viscosity solution
to \eqref{Eq-1}.
Then there exists a constant 
$C$, dependent only on $n, A_0, C_0, \diam(\Omega)$, radius of exterior
sphere and 
$\int_{\Rn}(|y|^2\wedge 1) \widehat{k}(y)\D{y}$, such that
\begin{equation}\label{EL2.1A}
|u(x)|\leq C K\delta(x)\quad \text{for all}\; x\in \Omega,
\end{equation}
where $\delta(x)=\dist(x, \Omega^c)$.
\end{lemma}

\begin{proof}
We first show that 
\begin{equation}\label{EL2.1B}
\abs{u(x)}\leq \kappa\, K \quad x\in \Rn,
\end{equation}
for some constant $\kappa$. From \cite[Lemma 5.5]{M19} we can find a non-negative function 
$\chi\in C^2(\bar{\Omega})\cap C_b(\Rn)$, with $\inf_{\Rn}\chi>0$, satisfying
$$
L \chi  + C_0 |D \chi |  \leq -1\quad \text{in}\;\,\, \Omega.
$$
Note that, since $\chi\in C^2(\bar{\Omega})$,
the above equation holds in the classical sense. 
Defining $\psi  = 2 (K+\varepsilon) \chi $, $\varepsilon>0$, we have that $\inf_{\Rn}\psi>0$ and
\begin{equation}\label{EL2.1C}
L \psi  + C_0 |D \psi |  \leq -2(K+\varepsilon) \quad \text{in}\;\,\, \Omega.
\end{equation}
We claim that $u \leq \psi$ in $\Rn$. Suppose, on the contrary, that $(u - \psi)(z) >0$ at some point in $z\in \Omega$.
Define
$$\theta = \inf \{ t : u \leq t+\psi \; \text{in}\; \Rn	\}.$$
Since $(u - \psi)(z) >0$, we must have $\theta\in (0, \infty)$.
Again, since $u=0$ in $\Omega^c$, there must
be a point $x_0 \in \Omega$ such that $u(x_0) = \theta + \psi (x_0)$ and $u \leq  \theta + \psi$ in $\Rn$.  Since $\psi$ is $C^2$ in $\Omega$, we get from the definition of viscosity  subsolution that 
\begin{align*}
-K\leq L(\theta + \psi)(x_0) + C_0 |D \psi (x_0)|
= L \psi(x_0) + C_0 |D \psi (x_0)| \leq -2(K+\varepsilon),	
\end{align*}
using \eqref{EL2.1C}. But this is a contradiction. 
This proves the claim that $u\leq \psi$ in $\Rn$.
Similar calculation using $-u$ will also give us $-u\leq \psi$
in $\Rn$. Thus 
\begin{equation*}
\abs u \leq 2\sup_{\Rn} |\chi|\, (K+\varepsilon) \quad \text{in}\; \Rn.
\end{equation*}
Since $\varepsilon$ is arbitrary, we get \eqref{EL2.1B}.

Now we can prove \eqref{EL2.1A}. In view of \eqref{EL2.1B}, it is
enough to consider the case $K>0$.
 Since $\Omega$ belongs to the
class $C^{2}$, it satisfies a uniform exterior sphere condition from outside. Let $r_\circ$ be a radius satisfying uniform exterior condition. From \cite[Lemma~5.4]{M19} there exists a bounded, Lipschitz continuous function
$\varphi$, Lipschitz constant being $r_\circ^{-1}$, satisfying

\begin{align*}
\varphi &= 0  \quad \text{in} \quad \bar{B}_{r_\circ} , 
\\
\varphi &> 0 \quad \text{in} \quad \bar{B}_{r_\circ}^c , 
\\
\varphi &\geq \varepsilon \quad \text{in} \quad  B_{(1+\delta)r_\circ}^c   , 
\\
L \varphi + C_0 |D \varphi|  &\leq -1   \quad \text{in} \quad B_{(1+\delta)r_\circ}\setminus \bar{B}_{r_\circ},
\end{align*}
for some constants $\varepsilon, \delta$, dependent on $C_0 , A_0$ and $\int_{\Rn}(|y|^2\wedge 1) \widehat{k}(y)\D{y}$.  Furthermore, $\varphi$
is $C^2$ in $B_{(1+\delta)r_\circ}\setminus \bar{B}_{r_\circ}$.
For any point $y\in\partial \Omega$,
we can find another point $z\in \Omega^c$ such that $\overline{B}_{r_\circ}(z)$ touches
$\partial \Omega$ at $y$. 
Let 
$w(x)=\varepsilon^{-1}\kappa K  \varphi(x-z)$.  Also $L(w) + C_0 | D w|  \leq -K $. Then
 $$
 L(u-w ) + C_0 | D(u-w)  |  \geq 0  \quad \text{in}\; B_{(1+\delta) r_\circ}(z)\cap \Omega.
$$  
Since, by \eqref{EL2.1B}, $u-w\leq 0$ in $(B_{(1+\delta) r_\circ}(z)\cap \Omega)^c$, from comparison principle \cite[Theorem~5.2]{BM21}
 it follows that 
$u(x)\leq w(x)$ in $\Rn$.  Note that the operator in \cite[Theorem~5.2]{BM21} does not have any gradient term, but the same proof (using the supersolution in \cite[Lemma~5.5]{M19}) also gives a comparison principle for the above operator.
Repeating a similar calculation for $-u$, 
we can conclude that $-u(x)\leq w(x)$ in $\Rn$.
This relation holds for any $y\in \partial \Omega$.
For any point $x\in \Omega$ with $\dist(x, \partial \Omega)< r_\circ$ we can find
$y\in \partial \Omega$ satisfying $\dist(x, \partial \Omega)=\abs{x-y}< r_\circ$. By previous estimate
we then obtain
$$|u(x)|\leq \varepsilon^{-1} \kappa K  \varphi(x-z)
\leq \varepsilon^{-1} \kappa K (\varphi(x-z)-
\varphi(y-z))\leq \varepsilon^{-1} \kappa K \,
r_{\circ}^{-1}\dist(x, \partial \Omega).
$$
This gives us \eqref{EL2.1A}.
\end{proof}
Now we are ready to prove that $u\in C^{0,1}(\Rn)$.
\begin{proof}[Proof of \cref{T2.1}]
Let $x \in \Omega$ and $r\in (0,1)$ be such that
$4r= \dist(x, \partial \Omega)\wedge 1$.  Without
loss of any generality, we assume $x=0$.
Define $v(y) = u(ry)$ in $\Rn$.
From \cref{L2.1},  we then have
   \begin{equation}\label{ET2.1B}
   | v(y) | \leq C_1\,K 
   \min\{r^{\alpha/2} (1+\abs{y}^{\alpha/2}), r (1+\abs{y})\} \quad \quad y \in \Rn,
   \end{equation}
for some constant $C_1$. We let
$$I_r f(x) = r^{\alpha}\frac{1}{2} \int_{\Rn} (f(x+y) + f(x-y) - f(x) ) k(ry) r^n \D{y},$$
and $L_r f= \Delta f + r^{2-\alpha}a\, I_r f$.
Let us compute $L_r v(x) + C_0 r |D v|$ in $B_2$.  Clearly, we have $\Delta v(x) = r^2 \Delta u(rx)$ and 
$Dv(x) = r Du(rx)$. Also
\begin{align*}
   r^{2-\alpha} I_r v(x) &= r^2\frac{1}{2} \int_{\Rn} (u(rx+ry) + u(rx-ry) - 2u(rx) ) k(ry) r^n \D{y} 
   \\
   &= r^2 I u(rx).
   \end{align*}
Thus, it follows from
\eqref{Eq-1} that 
\begin{equation}\label{ET2.1C}
\begin{split}
L_r v  + C_0 r |D v| &\geq -Kr^2 \quad \text{in} \quad B_2,
\\
L_r v - C_0 r|D v| &\leq Kr^2 \quad \text{in} \quad B_2.
\end{split}
\end{equation}
Now consider a smooth cut-off function $\varphi, 0\leq\varphi\leq 1$, satisfying
\begin{equation*}
    \varphi=
    \begin{cases}
    1\quad \text{in}\; B_{3/2},\\
    0\quad \text{in}\; B^c_2.
    \end{cases}
\end{equation*}
Let $w = \varphi v$. Clearly,  $((\varphi-1)v)(y) = 0$ for all $y \in B_{3/2}$, which gives $D((\varphi-1)v) = 0 $ and $\Delta ((\varphi-1)v)=0$ in $x \in B_{3/2}$. Since $w = v + (\varphi-1)v$, we obtain
\begin{equation}\label{ET2.1D}
\begin{split}
L_r w  + C_0 r |D w| &\geq -Kr^2 - A_0r^{2-\alpha}| I_r ((\varphi-1)v)| \quad \text{in} \quad B_1,
\\
L_r w - C_0 r|D w| &\leq Kr^2  + A_0 r^{2-\alpha}| I_r ((\varphi-1)v)| \quad \text{in} \quad B_1,
\end{split}
\end{equation}
from \eqref{ET2.1C}. Again, since $(\varphi-1)v=0$ in $B_{3/2}$,
we have in $B_1$ that
\begin{align*}
    |r^{2-\alpha} I_r ((\varphi-1)v)(x)|&= r^{n+2}\frac{1}{2}\Big|\int_{|y|\geq 1/2} ((\varphi-1)v)(x+y)+((\varphi-1)v)(x-y)      k(ry)\D{y}\Big|
    \\
 &\leq r^{n+2} \int_{|y|\geq 1/2} |v(x+y)| k(ry)\D{y}
 \\
 &\leq r^{n+2} \int_{\frac{1}{2}\leq\abs{y}\leq\frac{1}{r}} |v(x+y)| k(ry)dy +  r^{n+2} \int_{ r \abs{y} >1} |v(x+y)| k(ry)\D{y},   \\
&:= I_{r,1}+I_{r,2}.
\end{align*}
By \cref{Assmp-1}(a)
\begin{align*}
I_{r,1} = r^{n+2} \int_{\frac{1}{2}\leq\abs{y}\leq\frac{1}{r} } |v(x+y)| k(ry)\D{y} 
&\leq r^{2- \alpha} \int_{\frac{1}{2}\leq\abs{y}\leq\frac{1}{r}} 
|v(x+y)| \widehat{k}(y)\D{y} 
\\
&= \Lambda r^{2- \alpha} \int_{\frac{1}{2}\leq\abs{y}\leq\frac{1}{r}} |v(x+y)| \frac{1}{\abs y^{n+\alpha}}\D{y} 
\\
& \leq   r^{2- \alpha} \Lambda 3^{n+\alpha} \int_{|y|\geq 1/2}
\frac{|v(x+y)|}{ 1+ |x+y|^{\alpha/2}}   \frac{1+ |x+y|^{\alpha/2}}{1 + \abs y^{n+\alpha}}     \D{y} \\
    &\leq  \kappa_2 K r^{2- \alpha/2}  \int_{|y|\geq 1/2} 
    \frac{1+ |x+y|^{\alpha/2}}{ 1+ |x+y|^{n +\alpha}} \D{y}
\\
&\leq \kappa_3 K r^{2- \alpha/2},  
\end{align*}
for some constants $\kappa_2, \kappa_3$, and in the fifth line we use \eqref{ET2.1B}. Again, by \eqref{EL2.1B}, we have
\begin{align*}
I_{r,2}
\leq \kappa r^{2} K  \int_{ r \abs y >1 } r^n k(ry) \D{y}
&= \kappa r^{2}K  \int_{\abs y >1 }  k(y)\D{y}
\\
&\leq \kappa r^{2}K  \int_{\abs y >1 }  \widehat{k}(y)\D{y}
\leq \kappa r^{2} K \int_{\abs y >1 }  J(y)\D{y} \leq 
\kappa_4 r^{2}K,
\end{align*}
for some constant $\kappa_4$. Therefore, putting the estimates of 
$I_{1,r}$ and $I_{2, r}$ in \eqref{ET2.1D} we obtain
\begin{equation}\label{ET2.1E}
\begin{split}
L_r v  + C_0 r |D v| &\geq - \kappa_5 K r^{2-\alpha/2} \quad \text{in} \quad B_1,
\\
L_r v - C_0 r|D v| &\geq  \kappa_5 K r^{2-\alpha/2} \quad \text{in} \quad B_1,
\end{split}
\end{equation}
for some constant $\kappa_5$. Applying \cite[Theorem~4.1]{MZ21}
we obtain from \eqref{ET2.1E}
\begin{equation}\label{ET2.1F}
 \norm{v}_{C^{1} (B_{\frac{1}{2}}) } \leq \kappa_6 \Big( \norm{v}_{L^{\infty}(B_2)}+ r^{2-\alpha/2}  K \Big)   
\end{equation}
for some constant $\kappa_6$. The proof in \cite[Theorem~4.1]{MZ21}
is stated for equations but it is easily seen that the same proof also works for a system of inequalities as in \eqref{ET2.1E}. From
\eqref{ET2.1B} and \eqref{ET2.1F} we then obtain
\begin{equation}\label{ET2.1G}
\sup_{y\in B_{r/2}(x), y\neq x}\frac{|u(x)-u(y)|}{|x-y|}
\leq \kappa_7 K,
\end{equation}
for some constant $\kappa_7$.

Now we can complete the proof. Not that if $|x-y|\geq \frac{1}{8}$, then 
$$\frac{|u(x)-u(y)|}{|x-y|}\leq 2\kappa K,$$
by \eqref{EL2.1B}. So we consider $|x-y|< \frac{1}{8}$. If 
$|x - y| \geq 8^{-1} ( \delta(x) \vee \delta(y))$, then using
\cref{L2.1} we get
 $$
\frac{| u(x) - u(y)|  }{|x - y|}  
\leq 4CK (\delta(x)+\delta(y)) ( \delta(x) \vee \delta(y))^{-1}  \leq 8CK.
 $$
Now let $|x - y| < 8^{-1} \min\{ \delta(x) \vee \delta(y), 1\}$.
Then either $y\in B_{\frac{\delta(x)\wedge 1}{8}}(x)$ or $x\in B_{\frac{\delta(y)\wedge 1}{8}}(y)$. Without loss of generality, we suppose
$y\in B_{\frac{\delta(x)\wedge 1}{8}}(x)$. 
From \eqref{ET2.1G} we get
$$\frac{| u(x) - u(y)|  }{|x - y|}\leq \kappa_7 K.$$
This completes the proof.
\end{proof}

With the help of \cref{T2.1} we may choose $\beta=\infty$ in 
\cref{Assmp-1}(b).
\begin{remark}\label{R2.1}
Since $u$ is globally Lipschitz, choosing $\sJ(y)=|y|^{-n-\zeta}$,
$\zeta\in (0, 1\wedge \alpha)$,
we see from \cref{T2.1} that
\begin{equation}\label{ER2.1A}
|\int_{\Rn} (u(x+y)+u(x-y)-2u(x))\sJ(y)\D{y}|
\leq \kappa\norm{u}_{C^{0,1}(\Rn)}\leq \kappa C K
\end{equation}
for some constant $\kappa$. Let $\tilde{k}(y)=k(y)\Ind_{\{|y|\leq \beta^\prime\}} + \sJ(y)$, where $\beta^\prime<\frac{2}{3}\beta$.
It is easy to see that $\tilde{k}$ satisfies \cref{Assmp-1}(a).

We now show that \cref{Assmp-1}(b) also holds for this kernel
with $\beta=\infty$. 
Fix $r\in (0, 1]$ and $x_0 \in \Rn$ and choose
$x,y \in B_{\frac{r}{2}}(x_0)$ and $z\in B_r^c(x_0)$. Without
any loss of generality we may assume that given $\beta$
is in $(0,1/2)$. 
If $\max\{|x-z|, |y-z|\}\leq \beta^\prime$,
we have from \cref{Assmp-1}(b) that
$$\tilde{k}(x-z)=k(x-z) + \sJ(x-z) \leq \varrho k(y-z) +
3^{n+\zeta}\sJ(y-z)\leq (\varrho\vee 3^{n+\zeta})\tilde{k}(y-z),$$
using the fact 
$$\frac{1}{3}|x-z|\leq |y-z|\leq 3 |x-z|.$$
Also, if $|x-z|> \beta^\prime$, then
$$\tilde{k}(x-z)= \sJ(x-z)\leq 3^{n+\zeta}\tilde{k}(y-z).$$
Suppose $|x-z|\leq \beta^\prime$ and $|y-z|> \beta^\prime$. Note that
$|y-z|\leq \frac{3}{2}\beta^\prime<\beta$. Using \cref{Assmp-1}(a) we find
that
\begin{align*}
\tilde{k}(x-z)\leq \varrho k(y-z) + 3^{n+\zeta} \sJ(y-z)
&\leq \varrho\Lambda |y-z|^{-n-\alpha} + 3^{n+\zeta} \sJ(y-z)
\\
&\leq \left(\varrho\Lambda (\beta^\prime)^{-\alpha+\zeta}
+ 3^{n+\zeta}\right)\sJ(y-z)
\\
&= \left(\varrho\Lambda (\beta^\prime)^{-\alpha+\zeta}
+ 3^{n+\zeta}\right)\tilde{k}(y-z).
\end{align*}
Thus, the kernel $\tilde{k}$ satisfies \cref{Assmp-1}(b) for
$\beta=\infty$.

On the other hand, replacing the kernel $k$ by $\tilde{k}$ and
using \cref{T2.1}, \eqref{ER2.1A}, we obtain from \eqref{Eq-1} that
\begin{equation}\label{ER2.1B}
\begin{split}
\Delta u + aI_{\tilde k} u + C_0 |Du| & \geq - C_1 K \quad\text{in}\; \Omega,
\\
\Delta u + a I_{\tilde k} u - C_0 |Du| & \leq C_1 K \quad\text{in}\; \Omega,
\\
u & =0 \quad\text{in}\; \Omega^c,
\end{split}
\end{equation}
for some constant $C_1$, dependent on $\widehat{k}, \zeta, A_0$. This modification of nonlocal kernel would be useful apply the Harnack
inequality from \cite{F09}.
\end{remark}
\subsection{Fine boundary regularity}
In this section we study the regularity of $u/\delta$ in $\Omega$.
Since $u$ is Lipschitz, using the estimate \eqref{ET2.1A} we may
write \eqref{Eq-1} as follows
\begin{equation}\label{Eq-3}
|L u|=|\Delta u + a Iu|\leq CK\quad \text{in}\; \Omega,
\quad \text{and}\quad u=0\quad \text{in}\; \Omega^c,
\end{equation}
where $C$ is a constant depending on $\widehat{k}, A_0, C_0$. Also, in
view of \cref{R2.1}, we can assume that $k$ satisfies \cref{Assmp-1}(b)
for $\beta=\infty$. 
The rest of the section is devoted to the proofs of \cref{T2.2,T2.3}. Towards the proof of \cref{T2.2}, our first goal 
is to get the oscillation estimate \cref{P2.1}. To obtain this
result we need \cref{L2.2,L2.3,L2.4,L2.5,L2.6}.

\begin{lemma}\label{L2.2} 
There exists a constant $\tilde\kappa$, dependent on $n, A_0 ,\widehat{k}$,  such that for any $r\in (0, 1]$,  we have a bounded
radial function $\phi_r$ satisfying
\begin{equation*}
\begin{cases}
L\phi_r \geq 0   &\text{in} \;\; B_{4r} \setminus \bar{B}_r, 
\\
0 \leq \phi_r \leq \tilde\kappa r  & \text{in} \;\; B_r , 
\\
\phi_r \geq \frac{1}{\tilde\kappa} (4r - \abs x) & \text{in} \; \; B_{4r} \setminus B_r , 
\\
\phi_r \leq 0 & \text{in} \;\;  \Rn \setminus B_{4r}.
\end{cases}
\end{equation*}
Moreover, $\phi_r\in C^2(B_{4r}\setminus \bar{B}_r)$.
\end{lemma}

\begin{proof}
Fix $r\in(0, 1]$  and define $v_r(x) = \E^{- \eta q(x)} - \E^{-  \eta (4 r)^2}$, where $q(x) = |x|^2 \wedge 2(4 r)^2$ and
$\eta>0$.  Clearly, $1\geq v_r(0) \geq v_r(x)$ for all $x \in \Rn$.
Thus
\begin{equation}\label{EL2.2A}
v_r(x) \leq 1- \E^{-\eta(4r)^2}\leq \eta(4r)^2,
\end{equation}
using the fact that $1-e^{-s}\leq s$ for all $s\geq 0$.
Again, for $x \in B_{4r} \setminus B_r$, we have
\begin{align}\label{EL2.2B}
v_r(x) = \E^{-\eta (4r)^2 } ( \E^{\eta ((4r)^2 - q(x))} -1  )
&\geq \eta \E^{-\eta (4r)^2 }  ((4r)^2 - |x|^2) \nonumber
\\
& = \eta \E^{-\eta (4r)^2 }  (4r + |x|)  (4r - |x|) 
 \geq 5  \eta r  \E^{-\eta (4r)^2 }   (4r - |x|).
\end{align}
Now we estimate $Lv_r$ in $B_{4r} \setminus \bar{B}_r$. Fix 
$x\in B_{4r} \setminus \bar{B}_r$. Then
$$\Delta v_r = \eta \E^{- \eta \vert x  \vert^2} \left( 4\eta \vert x  \vert^2 - 2n \right),$$
and, since $Iv_r=I(v_r + e^{-\eta (4r)^2})$, using the convexity 
of exponential map we obtain
\begin{align*}
I(e^{-\eta q(\cdot)})(x)&\geq 
-\eta e^{-\eta|x|^2}\int_{\Rn} \Big( q(x+y) + q(x-y)- 2q(x)\Big) k(y) \D{y}
\\
&\geq -\eta e^{-\eta|x|^2} \left[ \int_{|y|\leq 1} \Big( q(x+y) + q(x-y)- 2q(x)\Big) k(y) \D{y} + \int_{|y|>1}  (8r)^2 k(y)\D{y} \right]
  \\
 &\geq -\eta e^{-\eta|x|^2} \left[\int_{\abs{y} <r}   
 \frac{ 2 \abs{y}^2}{\abs y^{n+\alpha}}  \D{y}
 + \int_{r<\abs{y}<1}   
 \frac{(8r)^2}{\abs y^{n+\alpha}}  \D{y}
  +(8r)^2\int_{|y|>1} J(y) \D{y} \right]
\\
&\geq -\eta e^{-\eta|x|^2} \kappa r^{2-\alpha},
\end{align*}
for some constant $\kappa$, independent of $\eta$. Combining the above estimates we see that, for $x\in B_{4r}\setminus\bar{B}_r$,
$$
Lv_r(x)\geq \eta \E^{- \eta \abs{x}^2} \Bigl[ 4  \eta \abs{x}^2 - 2n - A_0 \kappa r^{2-\alpha} \Bigr]\geq
\eta \E^{- \eta \abs{x}^2} \Bigl[ 4  \eta r^2 - 2n - A_0 \kappa r^{2-\alpha} \Bigr].
$$
Thus, letting $\eta= \frac{1}{r^2}(n+A_0\kappa)$, we obtain 
$$Lv_r>0\quad \text{in}\; B_{4r}\setminus\bar{B}_r.$$
We set $\phi_r= \frac{v_r}{r}$ and the result follows from \eqref{EL2.2A}-\eqref{EL2.2B}.
\end{proof}

Let us now define the sets that we use for our oscillation estimates. We borrow the notations of \cite{RS14}.

\begin{definition}\label{D2.1}
Let $ \kappa\in(0, \frac{1}{16})$ be a fixed small constant and let $ \kappa^{\prime} = 1/2 + 2\kappa$.
Given a point $x_0 \in \partial \Omega$  and $R>0$, we define
$$
D_R = D_R(x_0) = B_R(x_0) \cap \Omega,
$$
and 
$$
D^{+}_{\kappa^{\prime} R} = D^{+}_{\kappa^{\prime} R}(x_0) = B^{+}_{\kappa^{\prime} R} (x_0) \cap \left\lbrace  x \in \Omega : (x-x_0)  \cdot  {\rm n}(x_0) \geq 2\kappa R  \right\rbrace,
$$
where ${\rm n} (x_0)$ is the unit inward normal at $x_0$. Using the $C^{2}$ regularity of the domain, there exists $\rho>0$,  depending on $\Omega$,  such that the following inclusions hold for each $x_0 \in \partial \Omega$ and $R \leq \rho$:
\begin{equation}\label{E2.18}
B_{\kappa R}(y) \subset D_{R}(x_0) \qquad \text{ for all} \; y \in D^{+}_{\kappa^{\prime} R}(x_0),
\end{equation}
and 
\begin{equation}\label{E2.19}
B_{4\kappa R}(y^{\ast} + 4\kappa R {\rm n}(y^{\ast})) \subset D_R(x_0), \quad \text{and} \quad  
B_{\kappa R}(y^{\ast} + 4\kappa R {\rm n}(y^{\ast})) \subset D^{+}_{\kappa^{\prime} R}(x_0)
\end{equation}
for all $y\in D_{R/2}$, where $y^{\ast} \in \partial \Omega$ is the unique boundary point satisfying $|y-y^{\ast}| = \dist(y, \partial \Omega)$. Note that,  since $R\leq \rho$,  $y \in D_{R/2}$  is close enough to $\partial \Omega$ and hence the point 
$y^{\ast} + 4 \kappa R\, {\rm n}(y^{\ast})$ belongs to the line joining $y$ 
and $y^{\ast}$.
\end{definition}

\begin{remark}\label{R2.2}
In the remaining part of this section, we fix $\rho >0$ to be a small constant depending only on $\Omega$, so that \eqref{E2.18}-\eqref{E2.19} hold whenever $R \leq \rho$ and $x_0 \in \partial	\Omega$.
Also, every point on $\partial \Omega$ can be touched from both inside and outside $\Omega$ by balls of radius $\rho$. We
also fix $\upgamma>0$ small enough so that
for $0<r\leq \rho$  and $x_0 \in \partial \Omega$ we have 
$$
B_{\eta r}(x_0) \cap \Omega \subset B_{(1+\sigma)r} (z) \setminus \bar{B}_{r}(z)\quad \text{for}\quad \eta=\sigma/8, \, \sigma\in (0, \upgamma),
$$
for any $x^{\prime} \in \partial \Omega \cap B_{\eta r} (x_0)$, where
 $B_r (z) $ is a ball contained in $\Rn \setminus \Omega$ that touches $\partial \Omega$ at point $x^{\prime}$. 
\end{remark}

We first treat the case $\alpha\in (0,1)$. Note that in this situation $Iu$ can be defined in the classical sense and is bounded
in $\Omega$, by \cref{T2.1}.
\begin{lemma}\label{L2.3}
Let $\alpha\in(0,1)$ and  $\Omega$ be a bounded $C^{2}$ domain. Let $u$ be such that $u \geq 0$ in $\Rn$,  and $ \abs{L u} \leq C_2 $ in $D_R$, for some constant $C_2$. Then, there exists a positive constant $C$,  depending only on $n,  \Omega ,  A_0, \widehat{k}$, such that
\begin{equation}\label{EL2.3A}
\inf_{D^{+}_{\kappa^{\prime}R}} \left( \frac{u}{\delta} \right) \leq C \left( \inf_{D_{\frac{R}{2}}} \frac{u}{\delta} + C_2 R \right)
\end{equation}
for all $R \leq \rho_0$,  where the constant $\rho_0$ depends only on 
$n,  \Omega,  A_0$ and $\int_{\Rn}(|y|^2\wedge 1) \widehat{k}(y)\D{y}$.
\end{lemma}

\begin{proof}
We split the proof in two steps.

\noindent {Step 1.}\,  Suppose $C_2=0$ and $R \leq \rho$, where
$\rho$ is given by \cref{R2.2}.
Define $m= \inf_{D^{+}_{\kappa^{'}R} } u/ \delta \geq 0$.
By \eqref{E2.18}, 
\begin{equation}\label{EL2.3B}
u \geq m\delta \geq m\,(\kappa R)\quad  \text{in}\;\; D^{+}_{\kappa^{\prime}R}.
\end{equation}
Again, by \eqref{E2.19}, for any $y \in D_{R/2}$, we have either 
$y \in D^{+}_{\kappa^{\prime}R}$ or $\delta(y) < 4 \kappa R$.  If     $y \in D^{+}_{\kappa^{\prime}R}$ it follows from the definition of $m$ that $m \leq u(y)/ \delta(y)$. Now let $\delta(y) < 4 \kappa R$.
Let $y^{\ast}$ be the nearest point to $y$ on $\partial \Omega$ and
$ \tilde{y} = y^{\ast} + 4 \kappa R\, {\rm n} (y^{\ast})$.  Again by 
\eqref{E2.19},   we have $B_{4\kappa R}(\tilde{y})  \subset D_{R}  $ and $B_{\kappa R}(\tilde{y}) \subset D^{+}_{\kappa^{\prime}R} $.    Recall that $Lu = 0 $ in $D_R$  and $u \geq 0$ in $\Rn$.

Now take $r = \kappa R$ and let $\phi_r$ be the subsolution
in \cref{L2.2}. Define $\tilde{\phi}_r(x)  = \frac{1}{\tilde\kappa}\phi_r(x- \tilde{y})$. Using \eqref{EL2.3B} and
the comparison principle \cite[Theorem~5.2]{BM21} in 
$B_{4r}(\tilde{y}) \setminus \bar{B}_r (\tilde{y})$ 
it follows that $u(x) \geq m \tilde\phi_{r} (x)$ in all of $\Rn$.  In particular, we have $u / \delta \geq \frac{1}{(\tilde\kappa)^2} m$ 
on the segment joining $y^{\ast}$ and $\tilde{y}$, that contains $y$. Hence
$$
\inf_{D^{+}_{\kappa^{\prime}R}} \left( \frac{u}{\delta} \right) \leq C  \inf_{D_{\frac{R}{2}}} \frac{u}{\delta}\,.
$$

\noindent{Step 2.}\,  Suppose  $C_2>0$. Define $r^{\prime} = \eta r$ for $r \leq \rho$ and $\eta \leq 1$ to be chosen later.  Let  $\tilde{u}$ to be the solution of (cf. \cite[Theorem~1.1]{BM21})
$$
\begin{cases}
L \tilde{u} = 0  \quad \quad \text{in} \; D_{r^{'}} \,  ,  \\
\tilde{u} = u \quad \quad  \quad \; \, \text{in} \; \Rn \setminus D_{r^{'}} \, .
\end{cases}
$$
From step 1, we see that $\tilde{u}$
satisfies \eqref{EL2.3A}.  Define $w = \tilde{u} - u$.
Applying \cite[Theorem~5.1]{BM21}, we obtain that $\abs{Lw} \leq C_2$ in $D_{r^{\prime}}$ and $w =0$ in $\Rn\setminus D_{r^{\prime}}$. Since $r \leq \rho$,  points of $\partial\Omega$  can be touched by exterior ball of radius $r$. Thus for any point $y\in\partial\Omega$
we can find another point $z\in \Omega^c$ such that $\bar{B}_r(z)$ touches
$\partial \Omega$ at $y$. From the proof of \cite[Lemma~5.4]{M19}
there exists a bounded, Lipschitz continuous function $\varphi_r$, with Lipschitz constant  $r^{-1}$, that satisfies
$$
\begin{cases}
\varphi_r = 0, & \quad \text{in} \quad \bar{B}_r , 
\\
\varphi_r > 0, & \quad \text{in} \quad \bar{B}_r^c ,
 \\
L \varphi_r \leq - \frac{1}{r^2} , & \quad \text{in} \quad B_{(1+\sigma)r}\setminus \bar{B}_r,
\end{cases}
$$
for some constant $\sigma$, independent of $r$. Without any loss of
generality we may assume $\sigma\leq\upgamma$ (see \cref{R2.2}).
We set $\eta = \frac{\sigma}{8}$. Then  $D_{r^{\prime}} \subset B_{(1+\sigma)r}(z) \setminus  \overline{B}_r(z) $.  Letting 
$v(x) = C_2 r^2 \varphi_r (x - z)$ will give us a desired supersolution and therefore, by comparison principle we get 
$\abs{w} \leq v$ in $\Rn$.  For any point $x\in D_{r^{\prime}}$ we can find
$y\in \partial\Omega$ satisfying $\dist(x, \partial \Omega)=\abs{x-y}$.  By above estimate
we obtain
$$|w(x)|\leq  C_2 r^2  \varphi_r (x-z)
\leq C_2 r^2 (\varphi_r (x-z)-
\varphi_r (y-z))\leq C_2 r \dist(x, \partial \Omega) = C_2 r \delta(x).
$$
Thus we obtain
$$
|w(x)| \leq C_2 \frac{r^{\prime}}{\eta} \delta(x) \qquad \text{in} \;\; D_{r^{\prime}}.
$$
Combining with step 1 we have
$$
\inf_{D^{+}_{\kappa^{\prime} r^{\prime}}} \left( \frac{u}{\delta} \right) \leq \frac{C}{\eta} \left( \inf_{D_{\frac{r^{\prime}}{2}}} \frac{u}{\delta} + C_2 r^{\prime} \right).
$$
Setting $\rho_0=\eta\rho$ and $R=r^\prime$ we have the desired 
result.
\end{proof}

Next we obtain a similar estimate when $\alpha\in[1,2)$.
\begin{lemma}\label{L2.4}
Let $\Omega$ be a bounded $C^{2}$ domain and $u$ be such that $u\geq 0$ in all of $\Rn$ and $|L u| \leq C_2 g$ in $D_R$ for some positive constant $C_2$ and $g$ is given by
\begin{align*}
    g(x)=\begin{cases}
   ( \delta(x))^{1-\alpha} \quad &\text{if}\;  \alpha>1,
   \\
    -\log(\delta(x)) + C_3 \quad &\text{if}\; \alpha=1,
    \end{cases}
\end{align*}
for some constant $C_3$. Set $\hat\alpha=2-\alpha$ for $\alpha\in(1,2)$ and for $\alpha=1$, $\hat\alpha$ is any number
in $(0, 1)$.
Then there exists a positive constant 
$C$, depending on $\Omega , n $ and $\widehat{k}$, such that 
\begin{equation}\label{2.19}
    \inf_{D^+_{k^{\prime}R}} \frac{u}{\delta} \leq C \Big(\inf_{D_{R/2}}\frac{u}{\delta} + C_2 R^{\hat\alpha}\Big)
\end{equation}
for all $R<\rho_0,$ where $\rho_0$ is a positive constant depending only on $\Omega,  n, \hat\alpha,  A_0$ and $\int_{\Rn}(|y|^2\wedge 1) \widehat{k}(y)\D{y}$.
\end{lemma}

\begin{proof}
When $C_2=0$, the proof follows from Step 1 of \cref{L2.3}.
So we let $C_2>0$.
As before, we consider $\tilde{u}$ to be the solution of
\begin{align*}
    L \tilde{u}& =0\quad \text{in}\; D_R,
    \\
    \tilde{u} &=u\quad \text{in} \; \mathbb{R}^n\setminus D_R.
\end{align*}
Then 
$$
\inf_{D^+_{k^{\prime}R}}\frac{\tilde{u}}{\delta}\leq C\inf_{D_{R/2}}\frac{\tilde{u}}{\delta}
$$ 
holds, by step 1 of \cref{L2.3}. Defining $w=\tilde{u}-u$, 
we get $|L w|\leq C_2 g $ in $D_R$ by using \cite[Theorem~5.1]{BM21} and $w=0$ in $D^c_R$.
As before, we would consider an appropriate supersolution and then
apply comparison principle to establish \eqref{2.19}.

For this construction of supersolution
 we take inspiration from \cite[Lemma~5.8]{M19}. 
We set
$$
\tilde{\psi}(s)=\int_0^s 2\E^{-ql-q\int_0^l\Theta(\tau)\D\tau}
\D{l}-s,
$$ 
where  $q >0$ is to be chosen later and $\Theta$ is given by
$$
\Theta(s)=\int_{|z|>s}\min\{1, |z|\}\widehat{k}(z)\D{z}.
$$
Since $\Theta$ is integrable in a neighbourhood of $0$,
there exists a sufficiently small constant  $s(q) > 0$ such that, for $0<s<s(q)$,  
$\tilde{\psi}^{\prime}(s) = 2 e^{-qs-q\int_0^s\Theta(\tau)d\tau} -1 \geq \frac{1}{2}$.  Set $\sigma_1 = \min \{ \frac{s(q)}{8} , 1,\upgamma\}$. For any $r\in(0,1)$, we define
\begin{equation}\label{EL2.4A}
 \psi_{r ,z}(x) = 
 \begin{cases}
 \tilde{\psi} \left( \frac{d_{B_r(z)} (x)}{r}\right)   \quad &\text{ if} \; \;  d_{B_r(z)} (x) < r\sigma_1 ,
 \\[2mm]
 \tilde{\psi}(\sigma_1) \quad &\text{ if} \; \;  d_{B_r(z)} (x) \geq r\sigma_1,
 \end{cases}
 \end{equation}
where $d_{B_r(z)} (x) = \dist(x ,  B_r(z))$.
Let $\eta=\frac{\sigma_1}{8}$,  $0<r \leq \rho$ and $B_{\eta r}(x_0)\cap \Omega=D_{\eta r}$.  
We define
\[\Phi_r(x)=
\begin{cases}
\tilde{\psi}\left(\frac{\delta(x)}{r}\right), \quad &\text{ if} \; \;  \delta (x) < r\sigma_1,
\\[2mm]
 \tilde{\psi}(\sigma_1) \quad &\text{ if} \; \;  \delta (x) \geq r\sigma_1.
\end{cases}
\]
For $x \in D_{\eta r}$ then we have $x^{\ast}\in \partial \Omega$ such that $\delta(x)=|x-x^{\ast}|.$ Let $z^{\ast}_x=z$ be a point in $\Omega^c$ such that $B_r(z)$ touches $\partial \Omega$ at $x^{\ast}$. From \cref{R2.2} we have that 
$$
D_{\eta r} \subset   B_{(1+\sigma_1)r}(z)\setminus B_r(z).
$$
Since $\tilde\psi^{\prime\prime}<0$ and $|D \delta(x)|\geq \kappa>0$
for $\delta(x)\in (0, \rho_1)$, $\rho_1$ sufficiently small, it
follows that
\begin{align} \label{EL2.4C}
\Delta \Phi_r(x) \leq \frac{C}{r} + \tilde{\psi}^{\prime\prime} (\frac{\delta(x)}{r}) \frac{\kappa^2}{r^2}.
\end{align}
Consider $\psi_{r,z}$ from \eqref{EL2.4A} and notice that 
$\psi_{r, z}(x)=\Phi_r(x)$ and $\delta(x+y)\leq d_{B_r(z)}(x+y)$ for all $y \in \mathbb{R}^n$. Hence
\[
\psi_{r, z}(x+y)+\psi_{r, z}(x-y)-2\psi_{r, z}(x)\geq \Phi_r(x+y)+\Phi_r(x-y)-2\Phi_r(x).
\]
This readily gives (see \cite[Lemma~5.8]{M19})
\begin{equation}\label{EL2.4D}
I \Phi_r(x) \leq I \psi_{r, z}(x) \leq \frac{C}{r}\left(1+ \Theta\left(\frac{d_{B_r(z)}(x)}{r}\right)\right)
=\frac{C}{r}\left(1+ \Theta\left(\frac{\delta(x)}{r}\right)\right),
\end{equation}
using the fact $\delta(x)=|x-x^{\ast}|=d_{B_r(z)}(x)$.
Combining \eqref{EL2.4C} and \eqref{EL2.4D} we have
\begin{align*}
    L \Phi_r \leq \frac{C A_0}{r} \Big( 1 +  \Theta\Big(\frac{\delta(x)}{r}\Big) \Big) -\frac{2\kappa^2 q}{r^2}\Big(1+\Theta\Big(\frac{\delta(x)}{r}\Big)\Big).
\end{align*}
for all $x \in D_{\eta r}$. Now choose $q=\frac{1}{2\kappa^2}(C A_0+1)$ in the expression of $\tilde{\psi}$ we obtain
\begin{align}\label{EL2.4E}
  L \Phi_r \leq  -\frac{1}{r^2}\Big(1+\Theta\Big(\frac{\delta(x)}{r}\Big)\Big) \leq-\frac{1}{r^2} \Theta\Big(\frac{\delta(x)}{r}\Big),
\end{align}
for all $x \in D_{\eta r}$.

Next we estimate the function $\Theta$ in $D_{\eta r}$. 
For $\xi\in(0, 1]$, we see that
\begin{align}\label{AB1}
    \Theta(\xi)=\int_{|z|>\xi} \min\{1, |z|\}\widehat{k}(z)\D{z}
    &=\Lambda\int_{\xi <|z|\leq 1} \frac{|z|}{|z|^{n+\alpha}}\D{z}+\int_{|z|\geq 1}J(z)\D{z}\nonumber
    \\
    &=\Lambda \omega_n \int_\xi^1 \frac{r^n}{r^{n+\alpha}}\D{r}+ \kappa_1\nonumber
    \\
    &=\begin{cases}
    \frac{\omega_n \Lambda}{\alpha-1} \Big[\xi^{1-\alpha}-1\Big] + \kappa_1 \quad &\text{for}\; \alpha\in (1,2),
    \\
    \omega_n\Lambda (-\log\xi) + \kappa_1\quad &\text{for}\; \alpha=1,
    \end{cases}
\end{align}
for some positive constant $\kappa_1$. Here $\omega_n$ denotes the surface
area of the unit sphere in $\Rn$.
Since $ \frac{\delta(x)}{r}< \frac{1}{2}$ in $D_{\eta r}$, we get
from the above estimate that 
\begin{align*}
    \Theta\left(\frac{\delta(x)}{r}\right) &=\frac{\omega_n \Lambda}{\alpha-1} \Big[\left(\frac{\delta(x)}{r}\right)^{1-\alpha}-1\Big] + \kappa_1
    &\geq \Lambda \omega_n  \left(\frac{\delta(x)}{r}\right)^{1-\alpha}  \Big( \frac{ 2^{\alpha - 1} -1 }{2^{\alpha-1} (\alpha -1)}  \Big) \geq \kappa_2  \left[\frac{\delta(x)}{r}\right]^{1-\alpha}  . 
\end{align*}
for $\alpha\in (1,2)$, where the constant $\kappa_2$ is independent of $\alpha$.
Again, for $\alpha=1$, we have
\begin{equation}\label{EL2.4G}
\Theta\Big(\frac{\delta(x)}{r}\Big) =\Lambda \omega_n \log (\frac{r}{\delta(x)}) + \kappa_1
\end{equation}
for $x\in D_{\eta r}$.
We claim that for any $0<\zeta<1$, there exists a $r_\theta<1$ such that for all $r<r_\theta$
\begin{align}\label{EL2.4H}
   \log(rz) \geq r^\zeta \log(z)
\end{align}
for all $z \geq \frac{1}{\theta r}$, where $0<\theta<1$ is a fixed positive constant. To prove the claim, we let
$$
h(z)=\frac{\log(rz)}{\log(z)}.
$$ 
By our choice of parameters $z, r, \theta$, we have
$h(z)>0$. Since $\log z\geq \log(rz)$, we have
$$h^{\prime}(z)=\frac{(\log z-\log (rz))}{z(\log z)^2} >0$$
for $z>\frac{1}{\theta r}$. Thus $h$ is strictly increasing in
$[(\theta r)^{-1}, \infty)$, and therefore, 
\begin{align*}
    h(z)\geq h((\theta r)^{-1})=\frac{\log\Big(\frac{1}{\theta}\Big)}{\log\Big(\frac{1}{\theta r}\Big)}
    =\frac{\log\theta}{\log(r\theta)}\geq r^\zeta,
\end{align*}
for all $r\in (0, r_\theta)$, where $r_\theta$ depends only on
$\theta$ and $\zeta$. The gives us \eqref{EL2.4H}. Putting
\eqref{EL2.4H} in \eqref{EL2.4G} we have
\begin{align*}
 \Theta\Big(\frac{\delta(x)}{r}\Big)
  &\geq \Lambda \omega_n r^{\zeta} \log(\frac{1}{\delta(x)})  + \kappa_1 \geq C {r}^\zeta \left( \log\left(\frac{1}{\delta(x)}\right)  + \kappa_1 \right),
 \end{align*} 
in $D_{\eta r}$, for all $r\leq r_\theta$. Using the above
estimate and \eqref{EL2.4E}, we define the supersolutions as
\begin{equation*}
        v(x)= \upmu\, r^{1+\hat\alpha} \Phi_r(x), 
\end{equation*}
where the constant $\upmu$ is chosen suitably so that
$L v\leq -g$ in $D_{\eta r}$, for all $r\leq r_\theta$.
Thus, in $x\in D_{\eta r}$, we have $L(C_2 v)(x) \leq -C_2 g(x)$ and $C_2 v(x) \geq 0$ in $\mathbb{R}^n$.  Using comparison principle \cite[Theorem~5.2]{BM21} we then obtain $C_2 v(x) \geq w(x)$ in $\Rn$. Repeating the same argument with $-w,$ we get $|w(x)| \leq C_2 v(x)$ in $D_{\eta r}$. Now we can complete the proof repeating the same
argument as in \cref{L2.3}.
\end{proof}

\begin{lemma}\label{L2.5}
Let $\Omega$ be a bounded $C^{2}$ domain,  and 
$u$ be a bounded continuous function such that $u \geq 0$ in all of $\Rn$,  and $ \abs{L u} \leq C_2(1+\Ind_{[1,2)}(\alpha)g)$ in $D_R$, for some constant $C_2$. Let $\hat\alpha=1\wedge(2-\alpha)$
for $\alpha\neq 1$, and for $\alpha=1$, $\hat\alpha$ be
any number in $(0,1)$.
 Then, there exists a positive constant $C$,  depending only on $n, \Omega, A_0$ and $\widehat{k}$, such that
\begin{equation}\label{EL2.5A}
\sup_{D^{+}_{\kappa^{\prime}R}} \left( \frac{u}{\delta} \right) \leq C \left( \inf_{D^{+}_{\kappa^{\prime}R}} \frac{u}{\delta} + C_2 R^{\hat\alpha} \right)
\end{equation}
for all $R \leq \rho_0$,  where constant $\rho_0$ depends only 
on $\Omega , n ,  A_0, \hat{\alpha}$ and  $\int_{\Rn}(|y|^2\wedge 1) \widehat{k}(y)\D{y}$.
\end{lemma}

\begin{proof}
Recall from \cref{R2.1} that we may take $\beta=\infty$ in 
\cref{Assmp-1}(b). This property will be useful to
apply the Harnack inequality from \cite{F09}. We split the
proof in two steps.\\

\noindent{Step 1.}\, Let $C_2=0$. In this case \eqref{EL2.5A} follows from the Harnack inequality for $L$. Let $R \leq \rho$. Then for each $y \in D^{+}_{\kappa^{\prime}R}$ we have $B_{\kappa R}(y) \subset D_R $.  Hence we have $Lu = 0$ in $B_{\kappa R}(y)$.  Without loss of generality, we may assume $y=0$. Let $r=\kappa R$  and define  $v(x) = u(rx)$ for all $x \in \Rn$.  Then,
it can be easily seen that 
$$
r^2 Lu(rx)=  L_r v(x):=
\Delta v(x) + r^2 \frac{1}{2} \int_{\Rn} (v(x+y) + v(x-y) - 2v(x) ) k(ry) r^n \D{y}
 \qquad \text{for all } \; x \in B_1.
$$
This gives $L_r v(x) =  0$ in $B_1$ and $v \geq 0$ in whole $\Rn$.
From the stochastic representation of $v$ \cite[Theorem~1.1]{BM21},
it follows that $v$ is also a harmonic function in the probabilistic sense as considered in \cite{F09}.
Hence by the Harnack inequality \cite[Theorem~2.4]{F09} we obtain 
$$
\sup_{B_{\frac{1}{2}}} v \leq C \inf_{B_{\frac{1}{2}}} v,
$$
where constant $C$ does not depend on $r$. This of course, implies
$$
\sup_{B_{\frac{\kappa R}{2}}} u \leq C \inf_{B_{\frac{\kappa R}{2}}} u.
$$
Now cover $D^{+}_{\kappa^{\prime}R}$ by a finite number of balls $B_{\kappa R/2} (y_i)$, independent of $R$,  to obtain
$$
\sup_{D^{+}_{\kappa^{\prime}R}} u \leq C \inf_{D^{+}_{\kappa^{\prime}R} } u.
$$
\eqref{EL2.5A} follows since $ \kappa R/2 \leq \delta \leq 3 \kappa R/2$ in $D^{+}_{\kappa^{\prime}R}$.

\noindent{Step 2.}\,  Let $C_2>0$. The proof follows by
combining
Step 1 above and Step 2 of \cref{L2.3,L2.4}.
\end{proof}

Next we compute $L\delta$ in
$\Omega$.

\begin{lemma}\label{L2.6}
We have $|L\delta(x)|\leq C g(x)$, where $g$ is given by
\begin{equation}\label{EL2.6A}
  g(x)=\begin{cases}
   ( \delta(x) \wedge 1 )^{1-\alpha}\quad &\text{for}\quad \alpha>1,
   \\
    \log(\frac{1}{\delta(x) \wedge 1  }) + 1 \quad &\text{for}\quad \alpha=1,
    \\
    1 \quad &\text{for}\quad \alpha\in (0,1).
    \end{cases}
\end{equation}
\end{lemma}

\begin{proof}
Since $\delta\in C^{0,1}(\Rn)\cap C^{2}(\bar\Omega)$ 
\cite[Theorem~5.4.3]{DZ11}, \eqref{EL2.6A} easily follows for the case $\alpha\in (0,1)$.
Let $\Omega_{\rho_0}=\Omega \cap \{\delta < \rho_0\}$
where $\rho_0<1$. It
is enough to show that
\begin{equation}\label{EL2.6B}
|L \delta(x)| \leq C \Theta(\delta(x))
\quad \text{for}\;x \in \Omega_{\rho_0},
\end{equation}
where $\Theta$ is defined as before
$$
\Theta(\xi)=\int_{|z|>\xi}\min\{1, |z|\}\widehat{k}(z)\D{z}.
$$
First of all
\begin{align}\label{EL2.6C}
|L\delta(x)|  \leq |\Delta \delta(x)|+ A_0|I \delta(x)|
 \leq \kappa + A_0 |I \delta(x)|,
\end{align}
for some constant $\kappa$, depending on $\Omega$. Again,
\begin{align*}
I \delta(x) &=\int_{\mathbb{R}^n} \Big(\delta(x+z)+\delta(x-z)-2\delta(x) \Big)k(z)\D{z}
\\
&= \int_{|z|\leq\delta(x)/2} +\int_{|z|>\delta(x)/2} .
\end{align*}
Since $\delta(x+z)+\delta(x-z)-2\delta(x)\leq \kappa_2 |z|^2$
for $|z|\leq \delta(x)/2$, we have
$$\int_{|z|\leq \delta(x)/2} \Big(\delta(x+z)+\delta(x-z)-2\delta(x) \Big)k(z)\D{z}
\leq \kappa_2 \int_{|z|\leq \delta(x)/2}|z|^2\, \widehat{k}(z)\D{z}
\leq \kappa_3,$$
for some constant $\kappa_3$. Since $\delta$ is Lipschitz, it follows that
$$\delta(x+z)+\delta(x-z)-2\delta(x)\leq 2(\diam(\Omega)\vee 1)\min\{|z|, 1\}.$$
Thus
\begin{align*}
\int_{|z|>\delta(x)/2}\Big(\delta(x+z)+\delta(x-z)-2\delta(x) \Big)k(z)\D{z}
\leq \kappa_4 \int_{|z|>\delta(x)/2}\min\{|z|, 1\} \widehat{k}(z)
\D{z}=\kappa_4\Theta(\delta(x)/2),
\end{align*}
for some constant $\kappa_4$. Inserting these estimates in
\eqref{EL2.6C} we obtain
$$|L\delta(x)|\leq \kappa_5 (1+\Theta(\delta(x)/2))
\quad \text{for all}\; x\in\Omega_{\rho_0},$$
for some constant $\kappa_5$. Choosing $\rho_0$ sufficiently small,
\eqref{EL2.6B} follows from \eqref{AB1}.

\end{proof}

Now we are ready to prove our key estimate towards the regularity
of $u/\delta$.
\begin{proposition}\label{P2.1}
Let $u$ be a bounded continuous function such that $|Lu| \leq K$ in
$\Omega$, for some constant $K$, and $u=0$ in $\Omega^c$. Given any $x_0 \in \partial \Omega$, let $D_R$ be as in the \cref{D2.1}. Then for some $\uptau \in (0,1\wedge(2-\alpha))$ there exists $C$, dependent on 
$\Omega, n ,A_0, \alpha$ and $\widehat{k}$ but not on $x_0$, such that
\begin{equation}\label{EP2.1A}
    \sup_{D_R} \frac{u}{\delta} -\inf _{D_R} \frac{u}{\delta} \leq CK R^{\uptau}
\end{equation}
for all $R\leq \rho_0$, where $\rho_0>0$ is a constant depending only on $\Omega, n, \hat{\alpha},  A_0$ and $\int_{\Rn}(|y|^2\wedge 1) \widehat{k}(y)\D{y}$.
\end{proposition}

\begin{proof}
For the proof we follow a standard method, similar to \cite{RS14},
with the help of \cref{L2.4,L2.5,L2.6}.
Fix $x_0\in \partial \Omega$ and consider $\rho_0>0$ to be chosen later. With no loss of generality, we assume $x_0=0$.
In view of \eqref{EL2.1B}, we only consider the case $K>0$.
By considering $u/K$ instead of $u$, we may assume that $K=1$,
 that is, $|Lu| \leq 1$ in $\Omega$. From \cref{L2.1} we note that
$|u|_{C^{0,1}(\Rn)}\leq C_1$. For $\alpha\in (0, 1)$,
we can calculate $Iu$ classically and $|Iu|\leq \tilde{C}$ in $\Omega$, we can combine the nonlocal term on the rhs and only deal with $\Delta u$. In this case the proof is simpler and can be done
following the same method as for the local case (the proof below 
also works with minor modifications). Therefore, we only
deal with $\alpha\in [1,2)$.
 
We show that there exists $G>0,$ $\rho_1\in (0, \rho_0)$ and $\uptau \in (0,1)$, dependent only on $\Omega, n, A_0$ and $\widehat{k}$, and monotone sequences $\{M_k\}$ and $\{m_k\}$ such that, for all $k\geq 0,$
\begin{align}\label{EP2.1B}
    M_k-m_k=\frac{1}{4^{k\uptau}}, \,\,\,\,\, -1\leq m_k \leq m_{k+1}<M_{k+1}\leq M_k \leq 1,
\end{align}
and
\begin{align}\label{EP2.1C}
    m_k \leq G^{-1} \frac{u}{\delta} \leq M_k \quad \text{in} \quad D_{R_k}=D_{R_k}(x_0),  \quad \text{where}\quad R_k=\frac{\rho_1}{4^k}. 
\end{align}
Note that \eqref{EP2.1C} is equivalent to the following
\begin{align}\label{EP2.1D}
    m_k \delta \leq G^{-1} u \leq M_k \delta, \quad \text{in} \quad B_{R_k}=B_{R_k}(x_0),  \quad \text{where}\quad R_k=\frac{\rho_1}{4^k}.
\end{align}
Next we construct monotone sequences $\{M_k\}$ and $\{m_k\}$ by induction.

The existence of $M_0$ and $m_0$ such that \eqref{EP2.1B} and \eqref{EP2.1D} hold for $k=0$ is guaranteed by \cref{L2.1}.
Assume that we have the sequences up to $M_k$ and $m_k$. We want to show the existence of $M_{k+1}$ and $m_{k+1}$ such that \eqref{EP2.1B}-\eqref{EP2.1D} hold. We set
\begin{align}\label{EP2.1E}
    u_k=\frac{1}{G} u -m_k \delta.
\end{align}
Note that to apply \cref{L2.5} we need $u_k$ to be nonnegative in $\Rn$.
Therefore we work with $u^+_k$, the positive part of $u_k$. Let $u_k=u^+_k-u^-_k$ and by the induction hypothesis,
\begin{align}\label{EP2.1F}
    u^+_k=u_k\quad \text{and} \quad  u^-_k=0\quad \text{in} \quad B_{R_k}.
 \end{align}
We need a lower bound on $u_k$. Since $u_k\geq 0$ in $B_{R_k}$,
we get for $x\in B^c_{R_k}$ that
\begin{align}\label{EP2.1G}
u_k(x)=u_k(R_k x_{\rm u}) + u_k(x)-u_k(R_k x_{\rm u})
\geq -C_L |x-R_k x_{\rm u}|,
\end{align}
where $z_{\rm u}=\frac{1}{|z|}z$ for $z\neq 0$ and $C_L$ denotes a
Lipschitz constant of $u_k$ which can be chosen independent of $k$.
Using \cref{L2.1} we
 also have $|u_k|\leq G^{-1}+\diam(\Omega)=C_1$ for all
 $x\in\Rn$.
Thus using \eqref{EP2.1F} and \eqref{EP2.1G} we calculate
$Lu^-_k$ in $D_{\frac{R_k}{2}}$.
Let $x \in D_{R_k/2}(x_0)$. By \eqref{EP2.1F}, $\Delta u^-_k(x)=0$. 
Denote by
\[
\tilde{g}(r)=\left\{
\begin{array}{lll}
|\log(r)|\quad &\text{for}\; r>0,\; \alpha=1,
\\
r^{1-\alpha}\quad &\text{for}\; r>0,\; \alpha\in (1, 2).
\end{array}
\right.
\]
Then
\begin{align*}
 0\leq I u^-_k(x)&= \int _{x+y \not \in B_{R_k}} u^-_k(x+y)k(y)\D{y}
 \\
&\leq\int _{ \left\lbrace |y|\geq \frac{R_k}{2}, x+y\neq 0 \right\rbrace } u^-_k(x+y)k(y)\D{y}
 \\
&\leq C_L\int_{\left\lbrace \frac{R_k}{2}\leq |y| \leq 1,\; x+y\neq 0 \right\rbrace } \Big|(x+y)-R_k(x+y)_{\rm u}\Big|\widehat{k}(y)\D{y}
+ C_1\int_{|y|\geq 1} J(y)\D{y}
\\
 &\leq C_L  \int_{\frac{R_k}{2}\leq |y| \leq 1}(|x|+R_k) \frac{1}{|y|^{n+\alpha}}\, \D{y} +
 C_L  \int_{\frac{R_k}{2}\leq |y| \leq 1} \frac{1}{|y|^{n+\alpha-1}}\, \D{y} + \kappa_1 C_1
 \\
  &\leq \kappa_3 ((R_k)^{1-\alpha} + \tilde{g}({R_k}/{2}) + 1)
  \\
  &\leq \kappa_4 \tilde{g}(R_k)
    \end{align*}
for some constants $\kappa_1, \kappa_3, \kappa_4$, independent of 
$k$.

Now we write $u^+_k=G^{-1} u -m_k \delta + u^-_k$ and applying the operator $L,$ we get
\begin{align}\label{EP2.1H}
    |L u^+_k| &\leq G^{-1} |Lu| + m_k |L \delta| + |Lu^-_k|\nonumber
    \\
   & \leq G^{-1} + m_k C g(x) + \kappa_4 \tilde{g}(R_k),
\end{align}
using \cref{L2.6}.
Since $\rho_1\geq R_k \geq \delta$ in $D_{R_k}$, for $\alpha\geq 1,$ we have $R^{1-\alpha}_k \leq \delta^{1-\alpha}$, and hence, from \eqref{EP2.1H}, we have
\[|L u^+_k| \leq \Big[G^{-1} [\tilde{g}(\rho_1)]^{-1} + C + \kappa_4\Big] g(x): =\kappa_5 g(x) \quad \text{in}\quad D_{R_k/2}.\]
Now we are ready to apply \cref{L2.5}.
Recalling that
 \[u^+_k=u_k=G^{-1} u -m_k \delta \quad \text{in}\quad D_{R_k},\] we get from \cref{L2.4,L2.5} that
 \begin{equation}\label{EP2.1I}
 \begin{aligned}
     \sup_{D^+_{\kappa^{\prime} R_k/2}} \Big(G^{-1} \frac{u}{\delta}-m_k\Big) &\leq C\Big(\inf_{D^+_{\kappa^{\prime} R_k/2}}\Big(G^{-1} \frac{u}{\delta}-m_k\Big) + \kappa_5 R^{\hat\alpha}_k\Big)\\
     &\leq C\Big(\inf_{D_{ R_k/4}}\Big(G^{-1} \frac{u}{\delta}-m_k\Big) + \kappa_5 R^{\hat\alpha}_k\Big).
 \end{aligned}
 \end{equation}
Repeating a similar argument for the function $\tilde{u}_k=M_k\delta- G^{-1} u$, we obtain
 \begin{align}\label{EP2.1J}
     \sup_{D^+_{\kappa^{\prime} R_k/2}} \Big(M_k-G^{-1} \frac{u}{\delta}\Big)\leq C\Big(\inf_{D_{ R_k/4}}\Big(M_k-G^{-1} \frac{u}{\delta}\Big) + \kappa_5 R^{\hat\alpha}_k\Big)
 \end{align}
Combining \eqref{EP2.1I} and \eqref{EP2.1J} we obtain
     \begin{align}\label{EP2.1K}
         M_k-m_k &\leq C\Big(\inf_{D^+_{ R_k/4}}\Big(M_k-G^{-1} \frac{u}{\delta}\Big) +\inf_{D^+_{ R_k/4}}
         \Big(G^{-1} \frac{u}{\delta}-m_k\Big)+ \kappa_5 R^{\hat\alpha}_k\Big)\nonumber
         \\
         &=C\Big(\inf_{D_{ R_{k+1}}}G^{-1} \frac{u}{\delta}-\sup_{D_{ R_{k+1}}}G^{-1} \frac{u}{\delta} + M_k-m_k + \kappa_5 R^{\hat\alpha}_k\Big).
     \end{align}
Putting $M_k-m_k=\frac{1}{4^{\uptau k}}$ in \eqref{EP2.1K}, we have
 \begin{align}\label{EP2.1L}
         \sup_{D_{ R_{k+1}}}G^{-1} \frac{u}{\delta}-\inf_{D_{ R_{k+1}}}G^{-1} \frac{u}{\delta} &\leq \Big(\frac{C-1}{C}\frac{1}{4^{\uptau k}}+\kappa_5 R^{\hat\alpha}_k\Big)\nonumber
         \\
 &=\frac{1}{4^{\uptau k}} \Big(\frac{C-1}{C}+\kappa_5 R^{\hat\alpha}_k 4^{\uptau k}\Big).
  \end{align}
Since $R_k=\frac{\rho_1}{4^k}$ for $\rho_1\in (0, \rho_0)$,
we can choose $\rho_0$ and $\uptau$ small so that
\[
\Big(\frac{C-1}{C}+\kappa_5 R^{\hat\alpha}_k 4^{\uptau k}\Big) \leq \frac{1}{4^\uptau}.
\]
Putting in  \eqref{EP2.1L} we obtain 
 \[
 \sup_{D_{ R_{k+1}}}G^{-1} \frac{u}{\delta}-\inf_{D_{ R_{k+1}}}G^{-1} \frac{u}{\delta} \leq \frac{1}{4^{\uptau(k+1)}}.
 \]
Thus we find $m_{k+1}$ and $M_{k+1}$ such that \eqref{EP2.1B} and \eqref{EP2.1C} hold. It is easy to prove \eqref{EP2.1A} from
\eqref{EP2.1B}-\eqref{EP2.1C}.
\end{proof}

Now we are ready to complete the proof of \cref{T2.2}.
\begin{proof}[Proof of \cref{T2.2}]
As mentioned before, it is enough to consider \eqref{Eq-3}. Replacing $u$ by $\frac{u}{CK}$ we may assume that $|Lu|\leq 1$
in $\Omega$. Let $v=u/\delta$. From \cref{L2.1} we then have
\begin{equation}\label{ET2.2A0}
\norm{v}_{L^\infty(\Omega)}\leq C,
\end{equation}
for some constant $C$. Also, from \cref{T2.1} we have
\begin{equation}\label{ET2.2B}
\norm{u}_{C^{0,1}(\Rn)}\leq C.
\end{equation}
It is also easily seen that for any $x\in\Omega$ with $R=\delta(x)$
we have
$$\sup_{z_1,z_2\in B_{R/2}(x)}\frac{|\delta^{-1}(z_1)-\delta^{-1}(z_2)|}{|z_1-z_2|}\leq C R^{-2}.$$
Combining it with \eqref{ET2.2B} gives
\begin{equation}\label{ET2.2C}
\sup_{z_1,z_2\in B_{R/2}(x)}\frac{|v(z_1)-v(z_2)|}{|z_1-z_2|}\leq C (1+R^{-2}).
\end{equation}
Again, by \cref{P2.1}, for each $x_0 \in \partial \Omega$ and for all $r>0$ we have
\begin{equation}\label{ET2.2D}
\sup_{D_r(x_0)} v - \inf_{D_r(x_0)} v \leq C r^{\uptau}.
\end{equation}
where $D_r(x_0) = B_r(x_0) \cap \Omega$ as before. To complete the
proof it is enough to show that 
\begin{equation}\label{ET2.2E}
\sup_{x, y\in\Omega, x\neq y}\frac{|v(x)-v(y)|}{|x-y|^\upkappa}\leq C,
\end{equation}
for some $\eta>0$. Consider $x, y\in\Omega$ and let $r=|x-y|$.
We also suppose that $\delta(x)\geq \delta(y)$. If $r\geq 1/2$, then 
$$\frac{|v(x)-v(y)|}{|x-y|^\upkappa}\leq C 2^{1+\upkappa} ,$$
by \eqref{ET2.2A0}. So we suppose $|x-y|=r<1/2$. Let $R=\delta(x)$
and $x_0, y_0\in\partial\Omega$ satisfying $\delta(x)=|x-x_0|$
and $\delta(y)=|y-y_0|$. 
Fix $p>2$. Set $\kappa=[2+ \diam\Omega]^{-p}$. If $r\leq \kappa R^p$, then $r<\frac{1}{2}R$. In this case, it follows from \eqref{ET2.2C} that
$$|v(x)-v(y)|\leq C(1+ R^{-2})r\leq C(r + \kappa^{\nicefrac{2}{p}}r^{1-\nicefrac{2}{p}})\leq C_1 r^{1-\nicefrac{2}{p}}.
$$
Again, if $r\geq \kappa R^p$, we have $R\leq [r/\kappa]^{\frac{1}{p}}$. Thus, $y\in B_{\kappa_1 r^{\frac{1}{p}}}(x_0)$ for
$\kappa_1=1+\kappa^{-\nicefrac{1}{p}}$. From \eqref{ET2.2D} we then
have
$$|v(x)-v(y)|\leq C_2 r^{\nicefrac{\uptau}{p}}.$$
Thus \eqref{ET2.2E} follows by fixing $\upkappa=\min\{\frac{\uptau}{p}, 1-\frac{2}{p}\}$. This completes the proof.
\end{proof}

\begin{remark}\label{R2.3}
The regularity of $\partial\Omega$ in \cref{T2.2} can be relaxed to
$C^{1,1}$. In this case, $\delta$ will be a $C^{1,1}$ function. Therefore, $I\delta$ is defined classically and $L\delta$ can be
interpreted in the viscosity sense (see \cite{CCKS}). The proof of \cref{T2.2} goes through due to the coupling result in \cite[Lemma~5.2]{BM21}.
\end{remark}

\subsection{Boundary regularity of $D u$}
Using \cref{T2.2} we show that $D u \in C^{1, \gamma}(\Omega)$
for some $\gamma>0$. Recall \eqref{Eq-1}
\begin{equation}\label{E2.51}
\begin{split}
L u + C_0 |Du| & \geq -K \quad\text{in}\; \Omega,
\\
Lu - C_0 |Du| & \leq K \quad\text{in}\; \Omega,
\\
u & =0 \quad\text{in}\; \Omega^c.
\end{split}
\end{equation}
Let $v=u/\delta$. From \cref{T2.2} we know that $v\in C^{\upkappa}(\Omega)$. We extend $v$ in all of $\Rn$ as a $C^{\upkappa}$ function without altering its $C^{\upkappa}$ norm
(cf. \cite[Lemma~3.8]{RS14}).
Below we find the equations satisfied by
$v$.
\begin{lemma}\label{L2.7}
If $|Lu|\leq C$ in $\Omega$ and $u=0$ in $\Omega^c$, then we have
\begin{equation}\label{EL2.7A}
\frac{1}{\delta}[-C - v L\delta -Z[v,\delta]]\leq 
L v + 2\frac{D\delta}{\delta}\cdot D v
\leq \frac{1}{\delta}[C - v L\delta -Z[v,\delta]],
\end{equation}
in $\Omega$, where
$$
Z[v,\delta](x)=\int_{\Rn} (v(y)-v(x))(\delta(y)-\delta(x)) k(y-x)\D{y}.
$$
\end{lemma}

\begin{proof}
First of all, since $u\in C^1(\Omega)$ \cite[Theorem~4.1]{MZ21}, we have $v\in C^1(\Omega)$.
Therefore, $Z[v, \delta]$ is continuous in $\Omega$.
Consider a test function $\psi\in C^2(\Omega)$ that touches 
$v$ from above at $x\in\Omega$. Define
\[
\psi_r(z)=\left\{
\begin{array}{lll}
\psi(z)\quad &\text{in}\; B_r(x),
\\
v(z)\quad &\text{in}\; B^c_r(x).
\end{array}
\right.
\]
By our assertion, we must have $\psi_r\geq v$ for all $r$ small. To verify \eqref{EL2.7A} we must show that
\begin{align}\label{EL2.7B}
L \psi_r(x) + 2\frac{D\delta}{\delta}\cdot D \psi_r(x)
\geq \frac{1}{\delta(x)}[-C + v(x) L\delta(x) -Z[v,\delta](x)],
\end{align}
for some $r$ small. We define
\[
\tilde\psi_r(z)=\left\{
\begin{array}{lll}
\delta(z)\psi(z)\quad &\text{in}\; B_r(x),
\\
u(z)\quad &\text{in}\; B^c_r(x).
\end{array}
\right.
\]
Then, $\tilde\psi_r\geq u$ for all $r$ small. Since $|Lu|\leq C$
and $\delta\psi_r=\tilde\psi_r$, we obtain
\begin{equation*}
-C\leq L \tilde\psi_r(x)= \delta(x)L\psi_r(x) + v(x) L\delta(x) + 2D{\delta(x)}\cdot D\psi_r(x) + Z[\psi_r, \delta](x)
\end{equation*}
for all $r$ small. Rearranging the terms we have
\begin{equation}\label{EL2.7C}
-C-v(x) L\delta(x)- Z[\psi_r, \delta](x)\leq \delta(x)L\psi_r(x) + 2D{\delta(x)}\cdot D\psi_r(x).
\end{equation}
Let $r_1\leq r$. Since $\psi_r$ is decreasing with $r$, we get
from \eqref{EL2.7C} that
\begin{align*}
\delta(x)L\psi_r(x) + 2D{\delta(x)}\cdot D\psi_r(x)
&\geq \delta(x)L\psi_{r_1}(x) + 2D{\delta(x)}\cdot D\psi_r(x)
\\
&\geq \lim_{r_1\to 0}[-C-v(x) L\delta(x)- Z[\psi_{r_1}, \delta](x)]
\\
&= [-C-v(x) L\delta(x)- Z[v, \delta](x)],
\end{align*}
by dominated convergence theorem. This gives \eqref{EL2.7B}. Similarly we can verify the other side of \eqref{EL2.7A}.
\end{proof}

In order to prove \cref{T2.3}, we also need the following estimate on $v$. Define $\Omega_\sigma= \{x\in \Omega\; :\; \dist(x, \Omega^c)\geq \sigma\}$. Then we have

\begin{lemma}\label{L2.8}
For some constant $C$ it holds that
\begin{equation}\label{EL2.8A0}
\norm{D v}_{L^\infty(\Omega_\sigma)}\leq CK \sigma^{\upkappa-1}
\quad \text{for all}\; \sigma\in (0, 1).
\end{equation}
Furthermore, there exists $\eta\in (0,1)$ such that for any
$x\in \Omega_\sigma$ and $0<|x-y|\leq \sigma/8$ we have
$$\frac{|D v(y)- D v(x)|}{|x-y|^\eta}\leq CK \sigma^{\upkappa-1-\eta},$$
for all $\sigma\in (0,1)$.
\end{lemma}

\begin{proof}
As earlier, we suppose $K>0$.
Diving $u$ by $K$ in \eqref{E2.51} we may assume $K=1$. Using
\cref{T2.1} we can write $|Lu|\leq C_1$ in $\Omega$, for some 
constant $C_1$. By \cref{L2.7} we then have
\begin{equation}\label{EL2.8A}
\frac{1}{\delta}[-C_1 - v L\delta -Z[v,\delta]]\leq 
L v + 2\frac{D\delta}{\delta}\cdot D v
\leq \frac{1}{\delta}[C_1 - v L\delta -Z[v,\delta]],
\end{equation}
in $\Omega$.
Fix $x_0\in \Omega_\sigma$ and define
$$w(x)=v(x)-v(x_0).$$
From \eqref{EL2.8A} we then obtain
\begin{equation}\label{EL2.8B}
-\frac{1}{\delta}C_1 - \ell\leq 
L w + 2\frac{D\delta}{\delta}\cdot D w
\leq \frac{1}{\delta}C_1 - \ell,
\end{equation}
in $\Omega$, where
$$
\ell(x)= \frac{1}{\delta(x)}[w(x) L\delta (x)+Z[v, \delta](x)+ v(x_0) L\delta(x)].
$$
Set $r=\frac{\sigma}{2}$. We claim that
\begin{equation}\label{EL2.8C}
\norm{\ell}_{L^\infty(B_r(x_0)}\leq \kappa_1 \sigma^{\upkappa-2},
\quad \text{for all}\; \sigma\in (0, 1),
\end{equation}
for some constant $\kappa_1$.
Let us denote by
$$ \xi_1=\frac{w L\delta}{\delta},
\quad \xi_2=\frac{1}{\delta}Z[v, \delta]
\quad \text{and}\quad \xi_3= \frac{v(x_0)}{\delta}L\delta.$$
Recall that $\upkappa\in (0, (2-\alpha)\wedge 1)$ from \cref{T2.2}.
Since 
$$\norm{\Delta \delta}_{L^\infty(\Omega)}<\infty\quad
\text{and}\quad \norm{I \delta}_{L^\infty(\Omega_\sigma)}
\lesssim
\begin{cases}
   ( \delta(x) \wedge 1 )^{1-\alpha}\quad &\text{for}\quad \alpha>1,
   \\
    \log(\frac{1}{\delta(x) \wedge 1  }) + 1 \quad &\text{for}\quad \alpha=1,
    \\
    1 \quad &\text{for}\quad \alpha\in (0,1),
    \end{cases}
$$
(cf. \cref{L2.6}), and
$$\norm{v}_{L^\infty}(\Rn)<\infty, 
\quad \norm{w}_{L^\infty(B_r(x_0))}\lesssim r^{\upkappa},$$
it follows that
$$\norm{\xi_3}_{L^\infty(B_r(x_0))}\lesssim 
\left.\begin{cases}
\sigma^{-(\alpha\vee 1)} & \quad \text{for}\; \alpha\neq 1,
\\
\sigma^{-1}|\log\sigma|& \quad \text{for}\; \alpha= 1,
\end{cases}
\right\}
\lesssim 
\sigma^{-2 + \upkappa},$$
and
$$\norm{\xi_1}_{L^\infty(B_r(x_0))}\lesssim
\left.\begin{cases}
\sigma^{\upkappa-1+1-(1\vee \alpha)} & \quad \text{for}\; \alpha\neq 1,
\\
\sigma^{\upkappa-1}|\log\sigma|& \quad \text{for}\; \alpha= 1,
\end{cases}
\right\}
\lesssim \sigma^{-2+\upkappa}.$$
So we are left to compute the bound for $\xi_2$. Let $x\in B_r(x_0)$.
Denote by $\hat{r}=\delta(x)/4$. Note that 
$$\delta(x)\geq \delta(x_0)-|x-x_0|\geq 2r-r =r
\Rightarrow \hat{r}\geq r/4. $$
Thus, since $u\in C^1(\Omega)$ by \cite[Theorem~4.1]{MZ21}
(as mentioned before, the proof of \cite{MZ21} works for inequations),
$$|D v|\leq |\frac{D u}{\delta}|+ |\frac{uD \delta}{\delta^2}|\lesssim [\delta(x)]^{-1}\quad \text{in}\; B_{\hat{r}}(x),$$
using \eqref{EL2.1A}.
Since $\delta$ is Lipscitz and bounded in $\Rn$, we obtain
\begin{align*}
|Z[v, \delta](x)|&\leq
\Lambda \int_{y\in B_{\hat r}(x)}\frac{|\delta(x)-\delta(y)||v(x)-v(y)|}{|x-y|^{n+\alpha}} \D{y} + 
\int_{y\in B^c_{\hat r}(x)} |\delta(x)-\delta(y)||v(x)-v(y)| \widehat{k}(y-x)\D{y}
\\
&\lesssim [\delta(x)]^{-1} \int_{y\in B_{\hat r}(x)} |x-y|^{2-n-\alpha}\D{y}
+ \int_{y\in B_1(x)\setminus B_{\hat r}(x)}\frac{(\delta(x)-\delta(y))(v(x)-v(y))}{|x-y|^{n+\alpha}} \D{y}
\\
&\quad + \int_{|y|> 1} |\delta(x)-\delta(y+x)||v(x)-v(y+x)| J(y)\D{y}
\\
&\lesssim [\delta(x)]^{1-\alpha} + \int_{y\in B_1(x)\setminus B^c_{\hat r}(x)}\frac{|\delta(x)-\delta(y)||v(x)-v(y)|}{|x-y|^{n+\alpha}} \D{y}+
\kappa_2,
\end{align*}
for some constant $\kappa_2$.
The second integration on the right hand side can be computed as follows: for
$\alpha\leq 1$ we write
$$\int_{y\in B_1(x)\setminus B^c_{\hat r}(x)}\frac{|\delta(x)-\delta(y)||v(x)-v(y)|}{|x-y|^{n+\alpha}} \D{y}\lesssim 
\int_{y\in B_1(x)\setminus B^c_{\hat r}(x)} |x-y|^{-n-\alpha+1+\upkappa}\D{y}
\lesssim (1-\hat{r}^{1-\alpha+\upkappa})\lesssim \sigma^{-1+\upkappa},
$$
whereas for $\alpha\in (1,2)$ we can compute it as
$$\int_{y\in B_1(x)\setminus B^c_{\hat r}(x)}\frac{|\delta(x)-\delta(y)||v(x)-v(y)|}{|x-y|^{n+\alpha}} \D{y}
\lesssim \int_{y\in B_1(x)\setminus B^c_{\hat r}(x)} |x-y|^{-n-\alpha+1}
\lesssim \hat{r}^{-\alpha+1}\lesssim \sigma^{-1+\upkappa}.$$
Combining the above estimates we obtain
$$\norm{\xi_2}_{L^\infty B_r(x_0)}\lesssim \sigma^{-2+\upkappa}.$$
Thus we have established the claim \eqref{EL2.8C}.

Let us now define $\zeta(z)=w(\frac{r}{2}z + x_0)$. Letting
$b(z)=2\frac{D\delta(\frac{r}{2}z+x_0)}{\delta(\frac{r}{2}z + x_0)}$ and $r_1=\frac{r}{2}$ it follows from \eqref{EL2.8B} that
\begin{equation}\label{EL2.8D}
r^2_1\left(-\frac{C}{\delta}-\ell\right)\left(r_1z+x_0\right)
\leq
\Delta \zeta + r_1^{2-\alpha}I_{r_1} \zeta + 
r_1 b(z)\cdot D\zeta \leq r_1^2\left(\frac{C}{\delta}-\ell\right)\left(r_1z+x_0\right)
\end{equation}
in $B_2(0)$, where
$$I_{r_1} f(x) = r_1^{\alpha}\frac{1}{2} \int_{\Rn} (f(x+y) + f(x-y) - f(x) ) k(r_1y) r_1^n \D{y}.$$
Consider a cut-off function $\varphi$ satisfying
$\varphi=1$ in $B_{3/2}(0)$ and $\varphi=0$ in $B^c_2(0)$. 
Defining $\tilde\zeta=\zeta\varphi$ we get from \eqref{EL2.8D} that
$$
|\Delta \tilde\zeta(z) + r_1^{2-\alpha}I_{r_1}\tilde\zeta(z) + 
r_1 b(z)\cdot D\tilde\zeta|\leq
 r_1^2(\frac{C}{\delta}+|\ell|)(r_1z+x_0)+
r_1^{2-\alpha} |I_{r_1}((\varphi-1)\zeta)|
$$
in $B_1(0)$. Since 
$$\norm{r_1b}_{L^\infty(B_1(0))}\leq \kappa_3\quad \text{for all}\;
\rho\in(0,1),$$
applying \cite[Theorem~4.1]{MZ21} (this result works for inequality) we obtain,  for some $\eta\in(0,1)$,
\begin{equation}\label{EL2.8E}
\norm{D \zeta}_{C^\eta(B_{1/2}(0))}
\leq \kappa_6 \left(\norm{\tilde\zeta}_{L^\infty(\Rn)} + Cr_1
+ r_1^2\norm{\ell(r_1\cdot+x_0)}_{L^\infty(B_1)}
+ r_1^{2-\alpha} \norm{I_{r_1}((\varphi-1)\zeta)}_{L^\infty(B_1)}\right),
\end{equation}
for some constant $\kappa_6$ independent of $\rho\in (0,1)$.
Since $v$ is in $C^\upkappa(\Rn)$, it follows that
$$\norm{\tilde\zeta}_{L^\infty(\Rn)}=\norm{\tilde\zeta}_{L^\infty(B_2)}\leq \norm{\zeta}_{L^\infty(B_2)}\lesssim r^\upkappa.$$
Also, by \eqref{EL2.8C}, 
$$r_1^2\norm{\ell(r_1\cdot+x_0)}_{L^\infty(B_1)}\lesssim \sigma^{\upkappa}.$$

Note that, for $z\in B_1(0)$, 
\begin{align*}
|I_{r_1} (\varphi-1)\zeta|
&\leq \int_{|y|\geq 1/2} |(\varphi(x+y)-1)\zeta(x+y)|\widehat{k}(y) \D{y}
\\
&\leq 2\norm{v}_{L^\infty(\Rn)}\,\int_{|y|\geq 1/2} \widehat{k}(y) \D{y}
\\
&\leq \kappa_3
\end{align*}
for some constant $\kappa_3$. Putting these estimates in \eqref{EL2.8E} and calculating the gradient at $z=0$ we obtain
$$|D v(x_0)|\leq \kappa_4 \sigma^{-1+\upkappa},$$
for all $\sigma\in (0,1)$. This proves the first part.

For the second part, compute the H\"{o}lder ratio with
$D\zeta(0)-D\zeta(z)$ where
$z=\frac{2}{r}(y-x_0)$ for $|x_0-y|\leq \sigma/8$.
This completes the proof.

\end{proof}

Now we can establish the H\"{o}lder regularity of the gradient
up to the boundary.
\begin{proof}[Proof of \cref{T2.3}]
Since $u=v\delta$ it follows that
$$D u = vD \delta + \delta D v.$$
Since $\delta\in C^{2}(\bar\Omega)$, it follows from
\cref{T2.2} that
$v D\delta\in C^\upkappa(\bar\Omega)$. Thus, we only need to concentrate on $\vartheta=\delta D v$. Consider $\eta$ from
\cref{L2.8} and with no loss of generality, we may fix $\eta\in (0, \upkappa)$.

For $|x-y|\geq \frac{1}{8}(\delta(x)\vee\delta(y))$ it follows from
\cref{EL2.8A0} that
$$\frac{|\vartheta(x)-\vartheta(y)|}{|x-y|^\eta}
\leq C (\delta^\upkappa(x) + \delta^{\upkappa}(y))(\delta(x)\vee\delta(y))^{-\eta}\leq 2C.
$$
So consider the case $|x-y|< \frac{1}{8}(\delta(x)\vee\delta(y))$.
Without loss of generality, we may assume that
$|x-y|< \frac{1}{8}\delta(x)$. Then
$$\frac{9}{8}\delta(x)\geq |x-y|+\delta(x)\geq \delta(y)
\geq \delta(x)-|x-y|\geq \frac{7}{8}\delta(x).$$
By \cref{L2.8}, it follows
\begin{align*}
\frac{|\vartheta(x)-\vartheta(y)|}{|x-y|^\eta}
&\leq |D v(x)|\frac{|\delta(x)-\delta(y)|}{|x-y|^\eta}
+\delta(y)\frac{|D v(x)-D v(y)|}{|x-y|^\eta}
\\
&\lesssim \delta(x)^{-1+\upkappa}(\delta(x))^{1-\eta} +
\delta(y) [\delta(x)]^{\upkappa-1-\eta}
\\
&\leq C.
\end{align*}
This completes the proof by setting $\gamma=\eta$.
\end{proof}

\section{Overdetermined problems}\label{Overdetermined}

In this section we solve an overdetermined problem. Let $H:\RR\to\RR$ be a given locally Lipschitz function. We also assume 
that $a\in (0, A_0)$.
 Our main result of this section is the following.
\begin{theorem}\label{T3.1}
Let $\Omega \subset \Rn,$ $n\geq 2$, be an open bounded set with $C^{2}$ boundary. Suppose that \cref{Assmp-1} holds, 
$k=k(|y|)$ and $k:(0, \infty)\to (0, \infty)$ is strictly decreasing.
 Let $f:\RR\to\RR$ be locally Lipschitz and $u$ be a viscosity solution to
\begin{equation}\label{ET3.1A}
\begin{split}
Lu + H(|D u|) & =f(u)\quad  \text{in} \; \Omega,
\\
u&=0\quad \text{in} \;  \Omega^c,\quad u>0\quad \text{in}\; \Omega,
\\
\frac{\partial u}{\partial \rm n}&=c\quad \text{on}\; \partial \Omega,
\end{split}
\end{equation}
for some fixed $c>0$, where ${\rm n}$ is the unit inward normal on $\partial \Omega$. Then $\Omega$ must be a ball. Furthermore, $u$ is radially symmetric and strictly decreasing in the radial direction.
\end{theorem}

Proof of \cref{T3.1} follows from the boundary estimates in 
\cref{T2.2} combined with the approach of \cite{FJ15,BJ20}. 
Also, note that we have taken $k$ to be positive valued. This is just for convenience and the proofs below can be easily modified to 
include kernel $k$ that is non-increasing but strictly decreasing in a neighbourhood of $0$, provided we assume $\Omega$ to be connected.
We provide a sketch for the proof of \cref{T3.1} and the finer details can be found in \cite{FJ15,BJ20}. From \cref{T2.1,T2.3} we see that
$u\in C^{0,1}(\Rn)\cap C^{1, \gamma}(\bar\Omega)$, and therefore,
$u\in C^{2, \eta}(\Omega)$ by \cite{MZ21}. Therefore, we can assume that $u$ is a classical solution to \eqref{ET3.1A}.

Given a unit vector $e$, let us define the half space 
\[
\cH=\cH_{\lambda, e}=\{x\in \Rn : x\cdot e > \lambda\},
\]
and let $\bar{x}=\sR_{\lambda, e}(x)=x-2(x\cdot e)e+2\lambda e$ be the reflection of $x$ along $\partial \cH= \{x\cdot e=\lambda\}$. We say $v : \Rn \to \mathbb{R}$ is anti-symmetric if $v(x)=-v(\bar{x})$ for all $x\in \Rn$. Now let $D \subset \cH$ be a bounded open set and $u$ be a bounded anti-symmetric solution to
\begin{equation*}
\begin{aligned}
    Lu-\upbeta |D u| &\leq g \quad \text{in}\; D
    \\
    u & \geq 0\quad \text{in} \; \cH\setminus D,
\end{aligned}
\end{equation*}
where $\upbeta>0$ is a fixed constant.
Let
\[
 v=\begin{cases}
    -u &\text{in}\; \{u<0\}\cap D,
    \\[2mm]
    0 & \text{otherwise}.
    \end{cases}
\]
Then it can be easily seen that $v$ solves
\begin{equation}\label{E3.2}
Lv - \upbeta |D v|\geq -g\quad \text{in}\quad \Sigma:=\{u<0\}\cap D,
\end{equation}
 in the viscosity sense. To check \eqref{E3.2},
consider $x\in \Sigma$ and test function $\phi$ such that 
$\phi(x)=v(x)$ and $\phi(y)>v(y)$ for $y\in \Rn\setminus\{x\}$. Define $\psi:= \phi + (-u-v).$ Then $\psi(x)=-u(x)$ and $\psi(y)>-u(y)$ for $y\in \mathbb{R}^d\setminus\{x\}$. Furthermore, $\psi=\phi$ in $\Sigma$.
 Thus, we get $L\psi(x) + \upbeta |D \phi(x)| \geq -g(x)$. 
 This implies $L\phi + \upbeta|D \phi|+ L(-u-v)(x)+ 
 |D \phi(x)| \geq -g(x)$. Since $k$ is radially decreasing and
$u$ is anti-symmetric, it follows that $L(-u-v)(x)\leq 0$
(cf. \cite[p.~11]{BJ20}). This gives us \eqref{E3.2}.

The following narrow domain maximum principle is a
consequence of the ABP estimate in \cite[Theorem~A.4]{M19}.
\begin{lemma}\label{L3.1}
Let $\cH$ be the half space and $D\subset \cH$ be open and bounded. Also, assume $c$ to be bounded. Then there exists  a positive constant $C$, depending on $\diam D, n, k$, such that if $u \in C_b(\Rn)$ is an anti-symmetric supersolution of 
\begin{equation*}
\begin{aligned}
    Lu -\upbeta |D u|-c(x)u &=0\quad \text{in}\; D,
    \\
    u &\geq 0\quad \text{in} \; \cH\setminus D,
    \end{aligned}
\end{equation*}
then we have
\[
\sup_{\Omega} u^{-} \leq C \norm{c^{+}}_{L^{\infty}(D)}
\norm{u^{-}}_{L^n(D)}.
\]
In particular, given $\kappa_\infty>0$ and $c^{+}\leq \kappa_{\infty}$ on $D$, there is a $\delta>0$ such that if 
$|D|<\delta$, we must have $u\geq 0$ in $\cH$.
\end{lemma}

\begin{proof}
Set $\Sigma=\{u<0\}\cap D$ and define $v$ as above. From
\eqref{E3.2} we see that
\[
Lv +\upbeta |D v|+c(x)v \geq 0\quad \text{in}\; \Sigma.
\]
Since $v\geq 0$ in $\Sigma$ and $c\leq c^+$, we get 
$$Lv+\upbeta |D v|+c^+v \geq 0\quad \text{in}\; \Sigma.$$
Taking $f = -c^{+} v$ and using \cite[Theorem~A.4]{M19} we obtain,
for some constant $C_1$, that
\begin{align*}
\sup_{\Omega} u^{-}=
 \sup_{\Sigma}v &\leq \sup_{\Sigma ^c}\abs{v} + C_1 \norm{f}_{L^n(D)}\leq   C_1 \norm{c^+}_{L^{\infty}(D)}\norm{v}_{L^n(\Sigma)}
 =C_1 \norm{c^+}_{L^{\infty}(D)}\norm{u^-}_{L^n(D)}.
   \end{align*} 
This completes the proof of the lemma.
\end{proof}

Next result is a Hopf's lemma for anti-symmetric functions.
\begin{lemma}\label{L3.2}
Let $\cH$ be a half space, $D \subset \cH$, and $c\in L^{\infty}(D)$. If $u \in C_b(\Rn)$ is an anti-symmetric supersolution of $Lu -\upbeta |D u|-c(x)u=0$ in $D$ with $u \geq 0$ in $\cH$, then either $u\equiv 0$ in $\Rn$ or $u>0$ in $D$. Furthermore,
if $u\not\equiv0$ in $D$ and there exists a 
$x_0 \in \partial D\setminus\partial\cH$ with $u(x_0)=0$ such that there is a ball $B\subset D$ with $x_0 \in \partial B$, then there exists a $C>0$ such that
\[
\lim\inf_{t\to 0} \frac{u(x_0-t{\rm n})}{t} \geq C,
\]
where ${\rm n}$ is the inward normal at $x_0.$
\end{lemma}
\begin{proof}
With any loss of generality, we may assume that $c\geq 0$.
Suppose that $u\not\equiv 0$ and $u\ngtr 0$ in $D$. Then there exists a compact set $K \subset D $ such that $\inf_K u=\delta >0$ and a point $x_1 \in D$ such that $u(x_1)=0$. For $\varepsilon$ small
enough we can choose the test function
\[
\phi(y)=\left\{
\begin{array}{ll}
0 & \text{for}\; y\in B_\varepsilon(x_1),
\\[2mm]
u & \text{for}\; y\in B^c_\varepsilon(x_1).
\end{array}
\right.
\]
Note that $\Delta \phi(x_1)=0$,
$D\phi(x_1)=0$. Since $k$ is radially non-increasing and positive, from the proof of \cite[Theorem~3.2]{BJ20} it follows that $u\equiv 0$ in $D$. This contradicts our assertion. Thus
either $u\equiv 0$ in $\Rn$ or $u>0$ in $D$.

Now we prove the second part of the lemma. Assume that $u>0$ in $D$. Let $B$ be a ball in $D$ that touches $\partial D$ at $x_0$ and $B\Subset \cH$. This is possible since
$x_0\in \partial D\setminus \partial\cH$.
 Let $\vartheta$ be the positive solution to
$$L\vartheta-\upbeta |D \vartheta|=-1\quad \text{in}\; B, \quad
\text{and}\quad \vartheta=0\quad \text{in}\; B^c.$$
Existence of $\vartheta$ follows from \cref{T2.3} and Leray-Schauder fixed point theorem theorem.
Define $w=\kappa(\vartheta-\vartheta\circ\sR)$. 
Then we have $Lw \geq -\kappa C$ in $B$ for some positive constant $C$. Now repeating the arguments of \cite[Theorem~3.2]{BJ20} it
follows that for some $\kappa>0$ we have $u \geq w$ in $B$.
To complete the proof we need to apply Hopf's lemma on $\vartheta$ at the point $x_0$ (cf. \cite[Theorem~2.2]{BM21}).
\end{proof}

Given $\lambda \in \RR$,  $ e \in \partial B_1(0)$, define
\begin{equation}\label{E3.3}
v(x) = v_{\lambda,e} (x) = u(x ) - u(\bar{x})  \quad x \in \Rn,
\end{equation}
where $\bar{x} = \sR_{\lambda,e}(x)$ denotes the reflection of x 
by $T_{\lambda,e} \df \partial \cH_{\lambda,e}$ and $\cH_{\lambda,e} = \{ x\in \Rn \; : \; x \cdot e > \lambda\}$. 
We note that $\Rn \setminus \overline\cH_{\lambda,e} = \cH_{- \lambda, -e}$.  Moreover, let
$\lambda < l \df \sup_{x\in D} x\cdot e$.  Then $\cH \cap \Omega $ is nonempty for all $ \lambda < l$ and we put
$D_{\lambda} \df \sR_{\lambda ,e} (\Omega \cap \cH)$. Then for all $\lambda < l$ the function $v$ satisfies
\begin{equation}\label{E3.4}
\begin{split}
Lv + \upbeta |D v(x)| -c(x) v&\geq 0 \quad \text{in} \; D_{\lambda},
\\
Lv -\upbeta |D v(x)| -c(x) v&\leq 0 \quad \text{in} \; D_{\lambda},
\\
v &\geq 0\quad \text{in} \; \cH_{-\lambda, -e} \setminus D_{\lambda},
\\
v(x)&=-v(\bar{x}) \quad \text{for all}\; x \in \mathbb{R}^n, 
\end{split}
\end{equation}
where $\upbeta$ is the Lipschitz constant of $H$ on the interval
$[0, \sup |D u|]$ and
$$c(x)=\frac{f(u(x))-f(u(\bar{x}))}{u(x)-u(\bar{x})}\,.$$
 In view of Lemmas \ref{L3.1}-\ref{L3.2} and 
\eqref{E3.4}, we see that $v=v_{\lambda, e}$ is either $0$ in $\Rn$
or positive in $D_\lambda$ for $\lambda$ close to $l$. Now as we
decrease $\lambda$ one of following two situation may occur.
\begin{center}
{\bf Situation A}: there is a point $p_0\in 
\partial \Omega\cap\partial D_\lambda\setminus T_{\lambda, e}$,
\\
{\bf Situtation B}: $T_{\lambda, e}$ is orthogonal to $\partial \Omega$ at some point $p_0\in \partial \Omega\cap T_{\lambda, e}$.
\end{center}
$\lambda_0$ be the maximal value in $(-\infty, l)$ such that one of these situations occur. We show that $\Omega$ is symmetric with respect to $T_{\lambda_0, e}$. This would complete the proof of \cref{T3.1} since $e$ is arbitrary. Also, note that, since $u>0$  in $\Omega$, to 
establish the symmetry of $\Omega$ with respect to $T_{\lambda_0, e}$, it is enough to show that $v=0$ in $\Rn$. Suppose, to the contrary, that
$v>0$ in $D_{\lambda_0}$.\\[2mm]

\noindent{\bf Situation A:} In this case we have $v(p_0)=0$ and 
therefore, by \cref{T2.3,L3.2}, we get $\frac{\partial v}{\partial \rm n}(p_0)>0$. But, by \eqref{ET3.1A}, we have
$$\frac{\partial v}{\partial \rm n}(p_0)=\frac{\partial u}{\partial \rm n}(p_0) - \frac{\partial u}{\partial \rm n}(\sR(p_0))=0.$$
This is a contradiction.

\noindent{\bf Situation B:} This situation is a bit more complicated than the previous one. Set $T=T_{\lambda_0, e},
\cH=\cH_{\lambda_0, e}$ and $\sR=\sR_{\lambda_0, e}$. By rotation and translation, we may set $\lambda_0=0$, $p_0=0$, $e=e_1$
and $e_2\in T$ is the interior normal at $\partial D$.

Next two lemmas are crucial to get a contradiction in Situation B.
\begin{lemma}\label{L3.3}
We have
\[
v(t\bar{\eta })=o(t^2),\quad \text{as} \;\; t\to 0^+,
\] 
where $\bar{\eta}=e_2-e_1=(-1,1, 0, ...,0) \in \Rn.$
\end{lemma}

Proof of \cref{L3.3} follows from \cref{T2.2} and \cite[Lemma~3.2]{BJ20}.

\begin{lemma}\label{L3.4}
Let $D \subset \mathbb{R}^n, n\geq2$, be an open bounded domain such that $0\in \partial D$ and $\{x_1=0\}$ is orthogonal and there is a ball $B \subset D$ with $\bar{B}\cap {\partial\bar{D}}=\{0\}$. Denote
$$D^*:= D \cap \{x_1 <0\},$$
and assume that $w\in C_b(\mathbb{R}^n)$ is an anti-symmetric supersolution of
\begin{equation*}
\begin{aligned}
    Lw - \upbeta |D w|-c(x)w &=0 \quad \text{in} \; D^*
    \\
    w &\geq 0 \quad \text{in} \; \{x_1 <0\}
    \\
     w &>0 \quad \text{in}\; D^*.
 \end{aligned}   
\end{equation*}
Let $\bar{\eta}=e_2-e_1=(-1,1, \cdot \cdot \cdot,  0) \in \mathbb{R}^n$, then there exist positive $C, t_0$, dependent  on $D^*, n$, such that 
\[
w(t\bar{\eta}) \geq C t^2 
\]
for all $t\in (0, t_0).$
\end{lemma}
Clearly, \cref{L3.3,L3.4} give a contradiction to the Situation B.

\begin{proof}[Proof of \cref{T3.1}]
Proof follows from the above discussion and the arguments in
\cite[p.~11]{FJ15}.
\end{proof}
In the remaining part of this section we provide a proof of \cref{L3.4}.

\begin{proof}[Proof of \cref{L3.4}]
We follow the arguments of \cite[Lemma~3.3]{BJ20}.
Fix a ball $B=B_R(Re_2) \subset D$ for some $R>0$ small enough with $\partial B \cap \partial D=\{0\}$. Denote $$K=\{x_1<0\}\cap B.$$
Let $M_1 \Subset D^*$ such that $\theta\df\inf_{M_1} w >0$ and $M_2=\sR(M_1)$, that is, reflection of $M_1$ with respect to 
$\{x_1=0\}.$ Furthermore, we may assume that $M_1$ to be an open ball and taking $R$ smaller, we also assume that $\dist(K, M_1) >0$ and $|K| < \varepsilon$ for some small $\varepsilon >0$.  Now let $g$ be the unique positive viscosity solution to
\begin{align*}
    Lg-\upbeta |D g|&= -1\quad \text{in}\; B
    \\
    g&=0\quad \text{in}\; B^c.
\end{align*}
From \cref{T2.1}, we know that $\norm{g}_{C^{0,1}(\Rn)}\leq C$.
Let $\phi \in C^{\infty}_c(\mathbb{R}^d)$, $\supp(\phi) \subset M_1$, $0\leq \phi \leq 1$ and there exists a $U \subset M_1$ such that $\phi =1$ in $U$, $|U| >0$. Construct a barrier function $h$ of the following form:
\[h(x)=-\kappa x_1 g(x)+\theta \phi(x)-\theta \phi(\sR(x)).
\] 
Choosing $\kappa>0$ small enough, it can be easily checked that 
(see \cite[Lemma~3.3]{BJ20})
\begin{align*}
    Lh-\upbeta |D h| + c(x)h \geq 0
\end{align*}
in $K$. It is also standard to see that $g$ is radial about the point $Re_2$ (cf. \cite[Theorem~4.1]{BM21}).
Thus we have : (i) $w-h$ is anti symmetric, (ii) $w-h \geq 0$ in $\{x_1<0\}\setminus K$, since because $\theta >0$, and (iii) $L(w-h) -\upbeta|D(w-h)|-c(x)(w-h) \leq 0$. Applying \cref{L3.1}, we 
obtain $w-h\geq 0$ in $\{x_1<0\}$.
Hence 
\[w(t\bar{\eta})\geq h(t\bar{\eta})\geq C t^2\]
for $t\in (0, t_0)$, where we used Hopf's lemma on $g$. This completes the proof. 
\end{proof}
\subsection*{Acknowledgement}
We thank the referee for his/her careful reading of the manuscript and suggestions.
This research of Anup Biswas was supported in part by a Swarnajayanti fellowship (DST/SJF/MSA-01/2019-20). Mitesh Modasiya is partially supported by CSIR PhD fellowship (File no. 09/936(0200)/2018-EMR-I).

%

\end{document}